\documentclass[pdftex,12pt,a4paper]{article}
\usepackage[pdftex]{graphicx}
\usepackage[pdftex]{color}
\DeclareGraphicsExtensions{.png,.pdf,.eps}
\usepackage{thumbpdf}
\usepackage[font=footnotesize]{caption}
\usepackage{subcaption}

\graphicspath{{figs/}}
\usepackage[utf8]{inputenc}
\usepackage{amssymb}
\usepackage{amsmath}
\usepackage{amsthm}
\usepackage{url}
\usepackage[colorlinks=true]{hyperref}
\usepackage{bm}       

\newcommand{\norm}[1]{\left\lVert#1\right\rVert}
\newcommand{\N}{\mathbb{N}}
\newcommand{\R}{\mathbb{R}}
\newcommand{\C}{\mathbb{C}}
\renewcommand{\Re}{\mathop{\rm Re}\nolimits}

\newcommand{\Spec}{\mathop{\rm Spec}\nolimits}
\newcommand{\bmi}{{\bm i}}
\newcommand{\W}{{\mathcal{W}}}
\newcommand{\Wt}{{\tilde{W}}}

\newcommand{\enlletres}{%
   \renewcommand{\theenumi}{\alph{enumi}}%
   \renewcommand{\labelenumi}{(\theenumi)}%
}

\newtheorem{lemma}{Lemma}

\usepackage{authblk}

\title{Semi-analytical computation of heteroclinic connections between center
manifolds with the parameterization method}

\author[1]{Miquel Barcelona}
\author[2]{Alex Haro}
\author[1]{Josep-Maria Mondelo}

\affil[1]{Departament de Matemàtiques \& CERES-IEEC \& BGSMath, Universitat Autònoma de Barcelona. Av. de l’Eix Central, Ediﬁci C, Bellaterra (Barcelona) 08193, Spain}
\affil[2]{Departament de Matemàtiques i Informàtica \& Centre de Recerca Matemàtica \& BGSMath, Universitat de Barcelona, Gran Via 585, Barcelona 08007, Spain}

\date{}

\begin{document}

\maketitle

\begin{abstract} 
   This paper presents methodology for the computation of whole sets of
   heteroclinic connections between iso-energetic slices of center manifolds of
   center\, \linebreak[1]$\times$\, \linebreak[1]center\,
   \linebreak[1]$\times$\, \linebreak[1]saddle fixed points of autonomous
   Hamiltonian systems. It involves: (a) computing Taylor
   expansions of the
   center-unstable and center-stable manifolds of the departing and arriving
   fixed points through the parameterization method, using a new style that
   uncouples the center part from the hyperbolic one, thus making the fibered
   structure of the manifolds explicit; (b)
   uniformly meshing iso-energetic
   slices of the center manifolds, using a novel strategy that avoids numerical
   integration of the reduced differential equations and makes an explicit 3D
   representation of these slices as deformed solid ellipsoids;
   (c) matching the
   center-stable and center-unstable manifolds of the departing and arriving
   points in a Poincaré section. The methodology is applied to obtain the whole
   set of iso-energetic heteroclinic connections from the center manifold of
   $L_2$ to the center
   manifold of $L_1$ in the Earth-Moon circular, spatial Restricted Three-Body
   Problem, for nine increasing energy levels that reach the
   appearance of Halo
   orbits in both $L_1$ and $L_2$.  Some comments are made on possible
   applications to space mission design.

\end{abstract}

\section{Introduction}

Heteroclinic connections play an important role in the description of
dynamical systems from a global point of view. In this work we address the
systematic computation of heteroclinic connections between center manifolds of
fixed points of center$\times$center$\times$saddle type of autonomous,
3-degrees-of-freedom Hamiltonian systems.

Our interest in this case comes from applications in Astrodynamics,
namely to Libration Point (LP) space missions. In these missions,
spacecraft are sent to orbits that stay close to the
fixed (in a
rotating frame) points $L_1$, $L_2$ of the spatial, circular
restricted three-body problem (RTBP). This model describes the
dynamics of a massless particle under the attraction of two celestial
bodies, called primaries, that revolve in circles around their common
center of mass. Examples of primaries are the Sun and a planet, or a
planet and a moon. Among the advantages of these missions are the
absence of the shadow of a celestial body, thus providing a more stable
thermal environment, and continuous access to the whole celestial
sphere, except for a direction, that is not fixed but rotates with the
primaries. An actual example of the use of these advantages is found in
the orbit of the James Webb Space Telescope, around Sun-Earth
$L_2$, when compared to the orbit of the Hubble Space Telescope, around
Earth. See \cite{2003DuFa} for a survey of early LP missions, and the
webs of the space agencies for many newer ones.

From the point of view of the dynamics of the RTBP \cite{esa1, esa2,
esa3, esa4}, orbits in the center manifold of the $L_1,L_2$ points
 provide nominal
orbits for LP missions, whereas their invariant stable and unstable
manifolds provide transfer orbits either from Earth to a nominal orbit
or between nominal orbits. Focusing in this last case, missions that
have used trajectories close to heteroclinic connections are Genesis
\cite{1997HoBaWiLo} and Artemis\footnote{Not to be confused with the
human spaceflight Artemis program to return to the Moon.}
\cite{2011aSweetBroetal}. The main motivation of this paper is to
contribute towards the systematic design of similar missions.

For the Sun--Jupiter planar RTBP, a heteroclinic connection
between an object of the center
manifold of $L_1$ and an object of the center manifold of $L_2$, namely a
planar Lyapunov
periodic orbit (p.o.) in  each case, is
used in the celebrated paper \cite{2000KoLoMaRo} in order to explain the
(apparently) erratic behavior of
the comet Oterma. This heteroclinic connection is part of one of the
many families of  connections between Lyapunov
p.o.~explored in \cite{2006CaMa}. Some connections between Lyapunov
p.o.~have also been proven to exist through rigorous numerics
\cite{2003WiZg,2005WiZg}. The first work that addresses the
systematic computation of heteroclinic connections between center
manifolds of libration points, thus also obtaining connections between
invariant tori, is \cite{2000GoMa}. The
computations of this reference are used in
\cite{2004Goetal} in order to extend the results of
\cite{2000KoLoMaRo} to the spatial RTBP.
Reference \cite{2007aArMa} is
the first and only work (as far as we know) that computes the whole
set of heteroclinic connections between the center manifolds of $L_1$ and $L_2$ for
a fixed energy level. It does so not for the spatial RTBP
but for the spatial
Hill's problem. The methodology used is a combination of the
Lindstedt-Poincaré expansions introduced in \cite{2005Ma} with the
modification of the reduction to the center manifold technique
\cite{esa3,1999JoMa} used in \cite{2000GoMa}. The set of connections
found in \cite{2007aArMa} is shown in \cite{2008aDeMaRo}
to define a scattering map. 

In this paper we introduce a new technique for the computation of whole
sets of heteroclinic connections between iso-energetic slices of the center
manifolds of $L_1$ and $L_2$, that
is applicable to center$\times$center$\times$saddle fixed points of any
3-degrees-of-freedom, autonomous Hamiltonian system.  Compared to the one of
\cite{2007aArMa}, ours is a procedure that provides less insight on the fine
structure of the set of connections, but on the other hand it is more automatic
and
simpler to carry
out, and thus more adequate for the systematic exploration of phase space.
Our procedure relies solely on the semi-analytical computation of
center-unstable and center-stable manifolds. By avoiding the use of
Lindstedt-Poincaré expansions, it is also able to reach larger levels of
energy. 

The center-unstable and center-stable manifolds of $L_1$, $L_2$ are computed
through the parameterization method. The parameterization method, first
introduced in \cite{2003aCaFoLla,2003bCaFoLla,2005aCaFoLla}, has proven to be a
valuable tool in both the theoretical proof of existence of invariant manifolds
and their computation, due to the fact that the proofs are constructive and can
be turned into algorithms. Our approach is purely computational, and our
starting point is chapter 2 of \cite{mamotreto}. To the parameterization styles
described there, we add a mixed-uncoupling style that,
besides adapting the parameterization to the dynamics by making certain
sub-manifolds invariant, it is able to decouple the hyperbolic part from the
central one, thus making the fibered structure of the manifolds explicit. The
implementation of this new style of parameterization has been done starting from
the software from chapter 2 of \cite{mamotreto}, which is available at
\url{http://www.maia.ub.es/dsg/param/}.

A key point for our procedure to be systematic and to require little human
intervention is being able to produce equally-spaced meshes of iso-energetic
slices of the center manifold in a computationally efficient manner.
Classically, the center manifold of the collinear points of the spatial,
circular RTBP has been described globally through iso-energetic Poincaré
sections, that are two-dimensional and thus easily visualized
\cite{esa3,1999JoMa,2005GoMaMo}. A first strategy would be to globalize these
Poincaré sections to the iso-energetic slice through numerical integration of
the reduced equations. This turns out to be very expensive because the reduced
equations are given by high-order expansions of 4-variate functions. The
strategy we propose here avoids numerical integration completely, and also makes
explicit the representation of a 3D projection of the iso-energetic slice as a
deformed solid ellipsoid, that is what we actually mesh. On the topological
structure of iso-energetic slices of the center manifold, see
e.g.~\cite{1968Co,tesimcgehee,meyer-hall-offin,2004Goetal}.

The proposed methodology is applied in order to obtain the whole set of
heteroclinic connections from the center manifold of $L_2$ to the center
manifold of $L_1$\footnote{From these, the heteroclinic connections from the
   center manifold of
$L_1$ to the center manifold of $L_2$ can be obtained through one of the
symmetries of the RTBP (see eq.~\eqref{eq:symy}).} of the Earth-Moon spatial,
circular RTBP for 3 groups of 3 energy levels, each group of connections
performing a different number of revolutions around the Moon.  For the last
group of energies, the Halo family of p.o.~has already appeared, meaning that
these energies are relatively far from the equilibrium point.  We have chosen
Earth and Moon as primaries because of its intrinsic interest in space flight,
and also in order to see that the methodology works in a case in which, at
fixed energies, the dynamics around $L_1$ and $L_2$ are highly asymmetric one
with respect to the other.  

The paper is structured as follows: Section~\ref{sec:rtbp} introduces the
spatial, circular RTBP and provides several data related to the linear
dynamics around $L_1$, $L_2$. Section~\ref{sec:param} summarizes the
parameterization method for invariant manifolds of fixed points of flows, and
introduces our mixed-uncoupling style of parameterization for the computation
of center-stable and center-unstable manifolds. Section~\ref{sec:mesh}
introduces our meshing strategy for iso-energetic slices of the center
manifold. For the restrictions of
the center-unstable manifold of $L_2$ and the center-stable one of $L_1$
to the corresponding center manifolds, several meshes are computed and used
to test the validity of the expansions and decide the order to be used and the
energy range to be explored. Section~\ref{sec:conexs} describes the strategy
we follow to obtain the whole set of heteoclinic connections of an energy
level from meshes of the center
manifolds of $L_2$ and $L_1$. Finally,
in Section~\ref{sec:numres}, the whole set of heteroclinic connections is
computed for the 9 energy levels mentioned
previously.

\section{The spatial, circular RTBP}
\label{sec:rtbp}

In this section we recall some facts about the spatial, circular
Restricted Three-Body Problem (RTBP) that will be used in the following
sections. For more details, see e.g.~\cite{szebehely}.

The spatial, circular RTBP describes the motion of a particle of infinitesimal
mass which is considered to be under the gravitational influence of two other
massive bodies, known as primaries, that move in circular orbits around their
center of masses. Let us denote by $m_1$ and $m_2$ the masses of the primaries,
chosen in order to have $m_1 > m_2$.  By taking as origin the center of mass and
considering a uniformly rotating coordinate system with the same period as the
primaries, called synodic, the primaries can be made to remain fixed in the
horizontal axis. By further taking dimensionless distance and time units as to
have the distance between the primaries equal to $1$ and their period of
rotation equal to $2\pi$, the differential equations depend on a single
parameter $\mu=m_2/(m_1+m_2)\in[0,1/2)$. By finally introducing momenta as $p_x
= \dot{x} - y$, $p_y = \dot{y} + x$ and $p_z = \dot{z}$, the behavior of the
particle of infinitesimal mass is described by the following set of differential
equations, 
\begin{equation} 
   \begin{aligned} 
      \dot{x} &= p_x + y, \quad &\dot{p}_x &= p_y -
      \frac{1 - \mu}{r_1^3}(x-\mu) - \frac{\mu}{r_2^3}(x-\mu+1),\\ 
      \dot{y} &= p_y - x, \quad &\dot{p}_y &= -p_x - \frac{1-\mu}{r_1^3}y -
      \frac{\mu}{r_2^3}y,\\
      \dot{z} &= p_z,     \quad &\dot{p}_z &=
      -\frac{1-\mu}{r_1^3}z-\frac{\mu}{r_2^3}z, 
   \end{aligned}
   \label{eq:rtbp} 
\end{equation} 
with $r_1$ and $r_2$ defined as
$$
r_1 = \sqrt{(x-\mu)^2 +y^2+z^2}, \qquad r_2 = \sqrt{(x-\mu+1)^2+y^2+z^2}.
$$
This set of differential equations is a
Hamiltonian system with associated
Hamiltonian 
\begin{equation} \label{eq:hamrtbp}
   H(x,y,z,p_x,p_y,p_z) = \frac{1}{2}(p_x^2+p_y^2+p_z^2) -xp_y +
yp_x - \frac{1-\mu}{r_1} - \frac{\mu}{r_2}.  
\end{equation}
Due to symmetries in the differential equations \eqref{eq:rtbp}, if
$\bigl(x(t),y(t),z(t),p_x(t),p_y(t),p_z(t)\bigr)$ is a solution, then
\begin{equation}
   \bigl(x(-t),-y(-t),z(-t),-p_x(-t),p_y(-t),-p_z(-t)\bigr)\label{eq:symy}
\end{equation}
and
\[
   \bigl(x(t),y(t),-z(t),p_x(t),p_y(t),-p_z(t)\bigr)
\]
are also solutions.

It is well known that the RTBP has five equilibrium points. Two of them, called
Lagrangian and denoted by
$L_4,L_5$, form equilateral triangles with the primaries. The remaining three,
called collinear and denoted by $L_1,L_2,L_3$, lie on the line joining the
primaries, this is, the $x$ axis, see Figure~\ref{fig:rtbp}.
Their distances to the closest primary are given by
the only positive root of the corresponding Euler's quintic equation 
\[
   \begin{aligned} 
      \gamma_j^5 \mp (3-\mu)\gamma_j^4 + (3-2\mu)\gamma_j^3 - \mu \gamma_j ^2
      \pm 2\mu \gamma_j - \mu &= 0,  \quad &j&=1,2,\\
      \gamma_j^5 + (2+\mu)\gamma_j^4+(1+2\mu)\gamma_j^3 - (1-\mu)\gamma_j^2 -
      2(1-\mu)\gamma_j - (1-\mu) &= 0, \quad &j&= 3,  
   \end{aligned}
\]
so that, if we denote the $x$ coordinate of
the $L_j$ point as $x_{L_j}$, we have $x_{L_1}=\mu-1+\gamma_1$,
$x_{L_2}=\mu-1-\gamma_2$, $x_{L_3}=\mu+\gamma_3$. 

An important property of the collinear equilibrium points
is that their
associated set of eigenvalues has the following form
\[
   \Spec
   DF(L_{j}) = \{\bmi\omega_1^{(j)}, -\bmi\omega_1^{(j)},
   \bmi\omega_2^{(j)}, -\bmi\omega_2^{(j)},
   \lambda^{(j)}, -\lambda^{(j)}\}, \quad
   j=1,2,3, 
\]
where $\bmi$ denotes the imaginary unit, so their linear behaviour is
center$\times$center$\times$saddle. For this set of eigenvalues, the
corresponding eigenvectors will be denoted as
\[
   \begin{matrix}
      u^{(j)}, & \bar u^{(j)}, & v^{(j)},
      & \bar v^{(j)}, & w_+^{(j)}, & w_-^{(j)}, \\
   \end{matrix}
\]
where the overline denotes complex conjugation and
\begin{equation}\label{eq:eigvctrsri}
   u^{(j)} = u_r^{(j)}-\bmi u_i^{(j)}, \quad  v^{(j)}
   = v_r^{(j)}-\bmi v_i^{(j)}.
\end{equation}

\begin{figure}[htbp]
   \centering
   \includegraphics[width=.5\textwidth]{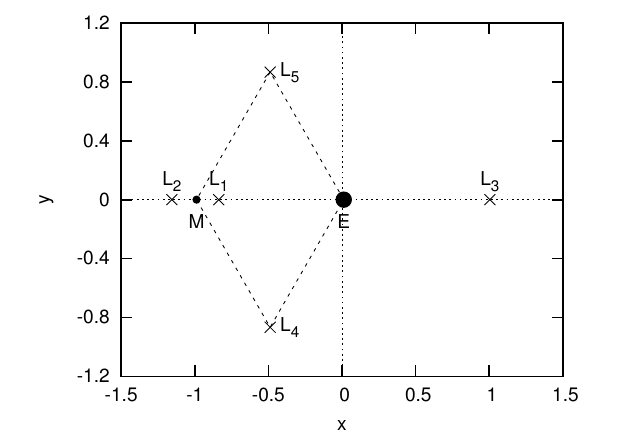}
   \caption{For the Earth-Moon mass
   parameter, scheme of the RTBP equilibrium points.}
   \label{fig:rtbp}
\end{figure}

From now on, $\mu$ will be the Earth-Moon mass ratio, taken as
\[
   \mu = 1.215058560962404 \cdot 10^{-2},
\]
as obtained from the DE406 JPL ephemeris file \cite{1998aSta}.

The eigenvalues in each one of
the two equilibrium points under study are numerically found to be
\begin{equation*} 
   \begin{aligned} 
      \omega_1^{(1)} &= 2.334385885086,\quad &\omega_2^{(1)} &=
      2.268831094972,\quad &\lambda^{(1)} &= 2.93205593364,\\ 
      \omega_1^{(2)} &= 1.862645862176513, \quad &\omega_2^{(2)} &=
      1.78617614289,\quad &\lambda^{(2)} &= 2.1586743203.
   \end{aligned} 
\end{equation*}
Note that, since neither $\omega_1^{(j)}$ is an integer multiple of $\omega_2^{(j)}$ nor 
$\omega_2^{(j)}$ is an integer multiple of $\omega_1^{(j)}$,
Lyapunov's center
theorem \cite{siegel-moser,meyer-hall-offin} applies to each pair of complex
eigenvalues, so two families of p.o.~are born from each collinear equilibrium
point. Each
one of these families spans a
2-dimensional manifold which is
tangent to the real and imaginary part of the eigenvectors associated to each
$\pm\bmi\omega_i$ eigenvalue at the equilibrium point. The planar
(resp.~vertical) family is tangent to eigenvectors with zero $z,p_z$
(resp.~$x,y,p_x,p_y$) components. From now on,
the
planar (resp.~vertical) eigenvalues will be denoted as $\pm\bmi\omega_p$
(resp.~$\pm\bmi\omega_v$),
in order to
emphasize their direction. According to the numerical values given above,
$\omega_p^{(j)}=\omega_1^{(j)}$ and
$\omega_v^{(j)}=\omega_2^{(j)}$.\footnote{Note that
Lyapunov's center theorem always applies in the planar RTBP. Therefore, the only
Lyapunov family whose existence can be compromised is the vertical one, in the
case that $\omega_p^{(j)}$ is an integer multiple of $\omega_v^{(j)}$. This does
not happen in the Earth-Moon case.}

\section{Parameterizing center-stable and center-unstable manifolds}
\label{sec:param}

For the purpose of computing all the possible heteroclinic connections from
the center manifold of $L_2$ to the one of $L_1$, to be denoted as
$\W^c(L_2)$, $\W^c(L_1)$ from now on, parameterizations of the
center-unstable manifold of $L_2$ and center-stable manifold of $L_1$ are
needed. The center-unstable manifold of $L_2$ (resp.~center-stable manifold of
$L_1$) will be denoted as $\W^{cu}(L_2)$ (resp.~$\W^{cs}(L_1)$) from now on.
Parameterizations of these manifolds are found by means of the
parameterization method of invariant manifolds of fixed points of flows
(PMFPF) as presented in chapter 2 of \cite{mamotreto}.  To the development of
\cite{mamotreto}, we add a mixed-uncoupling style, that effectively uncouples
the central and hyperbolic part in such a way that, for any point in a
center-stable or center-unstable manifold, the point of the center manifold at
the base of the corresponding fiber is obtained by setting to zero the last
parameter. In this section, we first recall the PMFPF briefly. After that, we
introduce our mixed-uncoupled style and discuss its application to the
computation of $\W^{cu}(L_2)$ and $\W^{cs}(L_1)$. We end the section with some
comments on how this mixed-uncoupled style makes the fibered structure of the
manifolds explicit.

Consider a $n$-dimensional dynamical system
\begin{equation}\label{eq:sysorigpm}
   \dot{Z} = F(Z),
\end{equation}
whose flow is $\Phi_t$ and
has a fixed
point $Z_0$, this is $F(Z_0)=0$.  We assume for simplicity that $DF(Z_0)$ has
$n$ distinct eigenvalues $\lambda_1,\dots,\lambda_n$, with corresponding
eigenvectors $v_1,\dots,v_n$, this is
\[
   DF(Z_0)v_j=\lambda_jv_j, \quad j=1, \dots ,n,
\]
with $\lambda_j\in\C$, $v_j \in \C^n$.  Our purpose consists in computing a
$d$-dimensional manifold containing $Z_0$ and tangent to the linear space
spanned by some of its eigenvectors.

For that, we consider the change of variables given by $Z =
Z_0 + Pz$, so that the new system of ODE $\dot{z} =
G(z)$ has the origin as a fixed point, and the differential of the vector
field evaluated at the new fixed point is diagonal with $DG(0) =
\text{diag} \{\lambda_1, \dots , \lambda_n \}$.  Note that $P$ is a matrix
composed by the eigenvectors in such a way that the first $d$ components
correspond to those eigendirections that we want our parameterization to be
tangent to.

Then, our problem turns into computing an expansion of a
$d$-dimensional manifold containing the origin, invariant by the flow described
by $G$ and tangent to the first $d$ components of $z$ at $z=0$. Then,
let us denote by $W : \C^d \to \C^n$ the parameterization
of the manifold and by $f : \C^d \to \C^d$
the reduced vector field\footnote{We will be looking for real manifolds.
Formally, we need to consider complex $W,f$, because $DG(0)$ may have complex
entries. The fact that $DG(0)$ is diagonal greatly simplifies the developments
that will follow. There is no computational overhead in the evaluation of the
complex $W$, $f$ thus obtained if the symmetries they have are exploited. See
\cite{mamotreto} for additional details.}, so that, if the manifold is
parameterized as $W(s)$, the differential equations reduced to the manifold are
$\dot s=f(s)$.  For this to be true, the following invariance equation must be
satisfied,
\[
   \psi_t\bigl(W(s)\bigr)=W\bigl(\phi_t(s)\bigr),
\]
where $\psi_t$ and $\phi_t$ are the flows of $\dot z=G(z)$ and $\dot s=f(s)$,
respectively. By differentiating with respect to $t$ and taking $t=0$, the
infinitesimal version of the invariance equation is obtained,
\begin{equation}\label{eq:invinf}
   G(W(s)) = DW(s)f(s),
\end{equation}
which is the equation that we actually solve for $W,f$.

The way of solving this invariance equation \eqref{eq:invinf} consists on
computing power series expansions for
both $W$ and $f$. Denote $W(s)=\bigl(W^1(s),\dots,W^n(s)\bigr)$,
$f(s)=\bigl(f^1(s),\dots,f^d(s)\bigr)$
and
\[
   W^i(s) = \sum_{k\geq0} W^i_k(s) ,
   \quad
   f^i(s) = \sum_{k\geq 0} f^i_k(s) ,
\]
where $W^i_k(s)=\sum_{|m|=k}W_m^i s^m$, $f_k^i(s)=\sum_{|m|=k}f_m^i s^m$,
$m=(m_1,\dots,m_d)\in\N^d$, $|m|=m_1+\dots+m_d$, $s=(s_1,\dots,s_d)$ and
$s^m=s_1^{m_1}\dots s_d^{m_d}$. Using these
expansions, equation \eqref{eq:invinf} is solved order by order.
Orders 0 and 1 are obtained as
\begin{equation} \label{eq:firstorders}
\begin{aligned}
   f_0(s) &= (0, \dots, 0), \quad &W_0(s) &= (0, \dots , 0), \\
   f_1(s) &= (\lambda_1 s_1, \dots, \lambda_d s_d ), \quad &W_1(s) &= (s_1, \dots, s_d, 0,
   \dots, 0).
\end{aligned}
\end{equation}
Imposing \eqref{eq:invinf} at order $k$ leads to
the \emph{$k$-th order cohomological equation},
whose right-hand side is given by all the known terms up to
order $k-1$,
\[
   R(s) = [G(W_{<k}(s))]_k - \sum_{l=2}^{k-1}DW_{k-l+1}(s)f_l(s), 
\]
where $W_{<k}(s)$ is the power series expansion of the manifold up to order
$k-1$ and $[\cdot]_k$ denotes ``terms of order $k$''.
In terms of this $R$, the cohomological equation becomes, at the
coefficient level,
\begin{equation}\label{eq:lolo}
   \begin{aligned}
      (\langle \lambda, m \rangle - \lambda_i) W_{m}^i + f_m^i &= R_m^i, \quad
      i \in \{1, \dots, d\} \\
      (\langle {\lambda}, m \rangle - \lambda_i) W_{m}^i &= R_m^i, \quad 
      i \in \{d+1, \dots, n\}.
   \end{aligned}
\end{equation}
Here we assume $R=(R^1,\dots,R^n)$, $R^i(s)=\sum_{|m|=k}R^i_m s^m$, and
$\lambda=(\lambda_1,\dots,\lambda_d)$, so that $\langle {\lambda}, m
\rangle = \lambda_1 m_1 + \dots + \lambda_d m_d$.

The first line of \eqref{eq:lolo} is known as the tangent part, while the
second line is called the normal part. For \eqref{eq:lolo} to have solution, it
is necessary that its normal part has no \emph{cross resonances}, this is,
$\langle {\lambda}, m \rangle\neq\lambda_i$ for all $m\in\N^d$ and
$i\in\{d+1,\dots,n\}$. In order to solve the tangent part, we can either
(a) take
$W^i_m=0$ and $f^i_m=R^i_m$ or (b) take $f^i_m=0$ and
$W^i_m=R^i_m/(\langle
{\lambda}, m \rangle-\lambda_i)$. For this last choice to be possible,
$(i,m)$ cannot be an \emph{inner resonance}, this is, $\langle
{\lambda}, m \rangle\neq \lambda_i$.
A \emph{style of parameterization} is a rule that determines the choice
between (a) and (b) as
a function of $(i,m)$. Some styles of parameterization are discussed in
\cite{mamotreto}.

For our purposes, it will be convenient to introduce a mixed-uncoupling style,
in which we require that several sub-manifolds of the form
$\{s_{i_1}=s_{i_2}=\dots=s_{i_J}=0\}$ are invariant, and we also require that the first $d-1$
equations of the reduced system of equations are uncoupled from the last one.
For a general $(i,m)$, in order to keep the parameterization simple, we will
want to choose (a) above unless it is incompatible with the invariance and
uncoupling requirements.
For the invariance requirement we follow \cite{mamotreto}: if
$I_1,\dots,I_N\subset\{1,\dots,d\}$ are such that $\{s_i=0,i\in I_j\}$,
$j=1,\dots,N$ are required to be invariant, we
choose (b)
if $\exists j: i\in
I_j$ and $m_l=0$ $\forall l\in I_j$.
For the uncoupling requirement, we need to eliminate the dependence of $f^1,\dots,f^{d-1}$ on $s_d$.
In order to achieve this, we
choose (b)
if $1\leq i \leq d-1$ and $m_d\neq 0$.
For both the sub-manifold invariance and uncoupling
requirements, it is necessary that $(i,m)$ is not an
internal resonance when choosing (b).

In our case, we want to compute $\W^{cu}(L_2)$ and $\W^{cs}(L_1)$ for the
Earth-Moon circular, spatial RTBP (see
Section~\ref{sec:rtbp}).
The choices for the ordering of eigenvalues are
\begin{equation}\label{eq:ordegvwcu}
   \lambda_{1,2}=\pm \bmi\omega_p^{(2)},\ 
   \lambda_{3,4}=\pm \bmi\omega_v^{(2)},\ \lambda_5>0
\end{equation}
for $\W^{cu}(L_2)$, and
\begin{equation}\label{eq:ordegvwcs}
   \lambda_{1,2}=\pm \bmi\omega_p^{(1)},\ 
   \lambda_{3,4}=\pm \bmi\omega_v^{(1)},\ \lambda_5<0
\end{equation}
for $\W^{cs}(L_1)$. We use the mixed-uncoupling style
just described with $d=5$, $N=3$, $I_1=\{5\}$, $I_2=\{3,4\}$, $I_3=\{1,2\}$.
This is, in addition to uncoupling the first $4$ equations from the $5$-th
one, we want the submanifolds $\{s_5=0\}$, $\{s_3=s_4=0\}$ and $\{s_1=s_2=0\}$
to be invariant. There are no internal resonances in
all the choices involved, namely:
\newcommand{\lolangle}{\langle\lambda,m\rangle}
\begin{enumerate}\enlletres
\item\label{en:nointres_a}
   For the uncoupling, we need $\lolangle - \lambda_i \neq0$ if $m_5\neq0$ and
   $i=1,\dots,4$. This is true
      because $\Re\lolangle\neq0$, since $\Re(\lambda_dm_d)\neq0$ and
   $\Re(\lambda_1m_1+\dots+\lambda_4m_4)=0$.
\item\label{en:nointres_b}
   For $\{s_5=0\}$ to be invariant, we need that $\lambda_5\neq\lolangle$ if
   $m_5=0$. This is necessarily true since
   $\Re(\lambda_1m_1+\dots+\lambda_4m_4)= 0$.
\item
   For $\{s_1=s_2=0\}$ to be invariant, we need that
   $\lambda_{1,2}\neq\lolangle$ if $m_1=m_2=0$. If $m_5\neq0$, it is true by the
   argument in (\ref{en:nointres_a}). If $m_5=m_1=m_2=0$, to have
   $\lambda_1=\lolangle$ or $\lambda_2=\lolangle$ would imply that
   $\omega_p^{(j)}$ is an integer multiple of $\omega_v^{(j)}$ which it is not
   true, by the hypothesis of Lyapunov's center theorem for the existence of the
   vertical Lyapunov family (see~Section~\ref{sec:rtbp}).
\item 
   For $\{s_3=s_4=0\}$, an analogous argument can be made\footnote{The
   invariance of $\{s_3=s_4=0\}$ does not actually need to be imposed: the
   terms $R^i_m$ are found to be zero when $m_3=m_4=0$ and $i=3,4$. This is a
   consequence of the fact that the planar RTBP is a subproblem of the spatial
   one.}.
\end{enumerate}

Following the notations of Section~\ref{sec:rtbp}, the RTBP equations
\eqref{eq:rtbp} are first diagonalized by choosing
\begin{equation}
\label{eq:pmatrix}
\begin{aligned}
   P &= \begin{bmatrix}
           u^{(2)} & \bar u^{(2)} & v^{(2)} & \bar v^{(2)}
               & w_+^{(2)} & w_-^{(2)}
        \end{bmatrix}, \quad\mbox{for $\W^{cu}(L_2)$},
        \\
   P &= \begin{bmatrix}
           u^{(1)} & \bar u^{(1)} & v^{(1)} & \bar v^{(1)}
               & w_-^{(1)} & w_+^{(1)}
        \end{bmatrix}, \quad\mbox{for $\W^{cs}(L_1)$}.
\end{aligned} 
\end{equation}
These choices lead to complex systems of ODE $\dot z=G(z)$, and therefore to
complex parameterizations. Following \cite{mamotreto}, the original real
center-unstable and center-stable manifolds are obtained, for real $s$, as
\begin{equation} \label{eq:canviparam}
   \tilde{W}(s) = Z_0 + P W (Cs),
\end{equation}
where $C$ is the $d\times d$ block-diagonal, square matrix
\begin{equation} \label{eq:cmatrix}
   C=
   \begin{pmatrix}
      1 & \bmi & 0 &  0 & 0 \\
      1 & -\bmi & 0 &  0 & 0 \\
      0 &  0 & 1 & \bmi & 0 \\
      0 &  0 & 1 & -\bmi & 0 \\
      0 &  0 & 0 &  0 & 1 \\
   \end{pmatrix}
   .
\end{equation}

The ordering of eigenvalues in
\eqref{eq:ordegvwcu} (resp.~\eqref{eq:ordegvwcs}) implies that $\{s_5=0\}$
describes $\W^c(L_{2})$ (resp.~$\W^c(L_{1})$), $\{s_3=s_4=0\}$ describes the
unstable (resp.~stable) manifold of the planar Lyapunov family of periodic
orbits, and $\{s_1=s_2=0\}$ describes the unstable (resp.~stable) manifold of
the vertical Lyapunov family of periodic orbits. By an intersection argument,
$\{s_3=s_4=s_5=0\}$ describes the planar Lyapunov family of p.o.~and
$\{s_1=s_2=s_5=0\}$ describes the vertical Lyapunov family of p.o.

Thanks to conditions (\ref{en:nointres_a}) and (\ref{en:nointres_b})
above,
parameter space can be considered a fibered space in which the base of each
fiber is described by the coordinates $s_1,\dots,s_4$, and $s_5$ is the
coordinate on each fiber. The reduced flow sends
fibers to fibers, since
\[
   \phi_t^i(s_1,\dots,s_4,s_5^{(1)})=
   \phi_t^i(s_1,\dots,s_4,s_5^{(2)}),
   \quad i=1,\dots,4,
\]
for all $s_1,\dots,s_4,s_5^{(1)},s_5^{(2)}$. Recall that we are computing
center stable (resp.~unstable) manifolds so any initial condition in the
manifold will tend to the center manifold forward (resp.~backward) in
time. Since the center manifold is described by $s_5=0$, this means that the
last component of the reduced flow will also
tend to 0. As a consequence, for $\tilde W (s_1,\dots,s_4,0)$ in any object of
the center manifold (fixed point, periodic orbit, invariant torus, etc.), the
corresponding fiber in the stable (resp.~unstable) manifold of the object is
given by $\tilde W(s_1,\dots,s_4,s_5)$ for $s_5\neq0$. Of course, $|s_5|$ has to
be small enough for the expansions to be accurate.

\section{Meshing iso-energetic slices of the center manifold}
\label{sec:mesh}

In this section, we present an strategy
to obtain equally-spaced meshes of iso-energetic slices of center manifolds of
center$\times$center$\times$saddle fixed points of
3-degrees-of-freedom, autonomous Hamiltonian systems. A first approach to deal
with this problem would consist
in globalizing by means of numerical integration the classical
representation of the center manifold
as a sequence of iso-energetic Poincaré sections \cite{esa3, 1999JoMa,
2005GoMaMo}. However, this procedure involves the evaluation of the
expansions obtained from the PMFPF, which is expensive in terms of computational
time. Instead of that, a
geometrical approach is taken in which, by considering the
quadratic part of the Hamiltonian applied to the parameterization, the
iso-energetic slice is described in terms of a deformed solid ellipsoid. The
strategy is used to perform an exploration in order to determine the order and
energy ranges to be used in the computations of the following sections.

Consider parameterizations $\tilde W^{(1)}(s)$ of $\W^{cs}(L_1)$ and $\tilde
W^{(2)}(s)$ of $\W^{cu}(L_2)$, $s \in \R^5$,
computed as described in the previous section. An iso-energetic slice of the
corresponding center manifold for an energy level $h$ is given by the set of
points defined by
\begin{equation}\label{eq:isoenslice}
   S_j^h = \{ \hat s \in \R^4 : H(\Wt^{(j)}(\hat{s},0)) = h\},
\end{equation}
for $\hat{s} = (s_1, \dots, s_4)^\top$ and $j$ denoting the libration point
$L_j$.  
Consider the eigenvectors that make up the matrix given by \eqref{eq:pmatrix}
and their corresponding real and imaginary parts as introduced in
\eqref{eq:eigvctrsri}. If they are chosen with suitable norms, the
matrix
\begin{equation} \label{eq:qmatrix}
   Q = 
   \begin{bmatrix}
      u_r & v_r & w_+ & u_i & v_i & w_-
   \end{bmatrix}
\end{equation}
can be made symplectic and can be used to define the following 
change of variables:
\begin{equation}\label{eq:canvisimp}
   Z = L_j + Qy.
\end{equation}
Since $\dot Z = F(Z)$ is a Hamiltonian system with Hamiltonian $H$ given by
\eqref{eq:hamrtbp}, the resulting system of ODE described by $\dot y = Y(y)$ is
also Hamiltonian with Hamiltonian
\begin{equation*}
   \tilde H(y) = H (L_j + Qy).
\end{equation*}

The change presented in \eqref{eq:canvisimp} transforms the linear part of the
new system, that is the differential of the original vector field evaluated at
the fixed point, into a realignment of its real Jordan form
\cite{perko96}. 
Since  
\begin{equation*}
      \dot y = \begin{pmatrix} 
         0 & 0 & 0 & \omega_p & 0 & 0 \\
         0 & 0 & 0 & 0 & \omega_v & 0 \\
         0 & 0 & \lambda & 0 & 0 & 0 \\
         -\omega_p & 0 & 0 & 0 & 0 & 0 \\
         0 & -\omega_v & 0 & 0 & 0 & 0 \\
         0 & 0 & 0 & 0 & 0 & -\lambda \\
      \end{pmatrix} y + \mathcal O (\norm{y}^2),
\end{equation*}
and 
$h_0:=\tilde H(0) = H(L_j)$, 
then the Hamiltonian is 
\[
   \tilde H(y) = h_0 + \frac{\omega_p}{2}(y_4^2 + y_1^2) +
   \frac{\omega_v}{2}(y_5^2+y_2^2) + \lambda y_3y_6 + \mathcal{O} (\norm{y}^3).
\]

By relating the matrices $P$
and $Q$, we will be able to provide an expression for the quadratic terms of
$H\bigl(\Wt(\hat s,0)\bigr)$. This is done in the following lemma.
\begin{lemma} \label{lemma}
   Assume that system \eqref{eq:sysorigpm} is
   Hamiltonian with Hamiltonian $H$, and consider the parameterization $
   \Wt \colon \R^5 \longrightarrow \R^6 $ where the matrix $P$
   in the change in \eqref{eq:canviparam} is given by any of
   the two matrices in \eqref{eq:pmatrix}. Then, for any $s=(\hat s, 0)$ it is satisfied
   \begin{equation}
      H (\Wt(\hat s, 0)) =
      h_0 + 2\omega_p
      (s_1^2+s_2^2)+2\omega_v(s_3^2+s_4^2) + \mathcal O (\norm{s}^3).
      \label{eq:hamsymp}
   \end{equation}
\end{lemma}
\begin{proof}
   Consider the change given by \eqref{eq:canviparam}. From \eqref{eq:cmatrix},
   \[
   C(\hat s, 0)^\top = (s_1+\bmi s_2, s_1 - \bmi s_2,
                        s_3+\bmi s_4, s_3-\bmi s_4,0)^\top.
   \]
   On the other hand, from
   \eqref{eq:firstorders}
   \[
      W(s) = W_1(s) + \mathcal O (\norm{s}^2) = (s_1, \dots, s_d, 0, \dots,
      0)^\top +
      \mathcal O (\norm{s}^2).
   \]
   Therefore,
   \[
   \begin{split}
      P W (C (\hat s, 0)^\top)
         &= P W_1(s_1+\bmi s_2, s_1 - \bmi s_2,
                  s_3+\bmi s_4, s_3-\bmi s_4,0)^\top
            + \mathcal O (\norm{s}^2)\\
      &=
      \begin{bmatrix}
         u & \bar u & v & \bar v & w_+ & w_-
      \end{bmatrix}
      (s_1+\bmi s_2, s_1 - \bmi s_2,
       s_3+\bmi s_4, s_3-\bmi s_4,0,0)^\top  + \mathcal O (\norm{s}^2)\\
      & = 2 \Re \big((u_r-\bmi u_i)(s_1+\bmi s_2)\big)
            + 2\Re\big((v_r-\bmi v_i) (s_3 + \bmi s_4)\big)
            + \mathcal O (\norm{s}^2) \\
      & = 2  
      \begin{bmatrix}
         u_r & v_r & w_+ & u_i & v_i & w_- 
      \end{bmatrix}
      (s_1, s_3, 0 , s_2, s_4, 0 )^\top + \mathcal O (\norm{s}^2) \\
      & = 2 Q \sigma(\hat s) + \mathcal O (\norm{s}^2),
   \end{split}
   \]
   where $Q$ is the defined in \eqref{eq:qmatrix} and
   $\sigma \colon \R^4 \longrightarrow \R^6$ is given
   by $\sigma (\hat s) = (s_1, s_3, 0, s_2, s_4 , 0)^\top$. Then, 
   \[
      \begin{split}
         H (\Wt(\hat s, 0)) &= H(L_j + P W (C (\hat s, 0)^\top)) \\
         &= H(L_j + 2 Q \sigma(\hat s) + \mathcal O
      (\norm{s}^2)) \\
         & = H(L_j + Q(2\sigma(\hat s))) + \mathcal O (\norm{s}^3) \\
         & = \tilde H (2 \sigma(\hat s)) = h_0 + 2\omega_p
         (s_1^2+s_2^2)+2\omega_v(s_3^2+s_4^2) + \mathcal O (\norm{s}^3), \\
      \end{split}
   \]
   as stated.
\end{proof}

Suppose that we want to mesh the iso-energetic slice
\eqref{eq:isoenslice} in $s_1$, $s_2$ and $s_4$ and recover the value
of $s_3$ from the energy. Disregarding cubic terms in
\eqref{eq:hamsymp},
\[
   h - h_0 = 2\omega_p (s_1^2+s_2^2)+2\omega_v(s_3^2+s_4^2),
\]
from which it is clear that the maximum value that can take $s_3^2$ is
$(h-h_0)/2\omega_v$.
Then, still disregarding cubic
terms in \eqref{eq:hamsymp}, the projection over
$(s_1,s_2,s_4)$ of $S_j^h$ would be 
\begin{equation}\label{eq:solidellipsoid}
   h \geq h_0 + 2\omega_p (s_1^2+s_2^2)+2\omega_v s_4^2,
\end{equation}
that is a solid ellipsoid whose boundary corresponds to
equality in this last equation.
Actually, since for $(s_1,s_2,s_4)$ in the interior of the ellipsoid we
have two solutions for $s_3$, we can consider that $S^h_j$ can be represented
in $s_1,s_2,s_4$ as two solid ellipsoids glued by its boundary.

However, since the original equation presents terms of order $\mathcal O_3$,
the projection in $s_1,s_2,s_4$ of the iso-energetic slice is not
the solid ellipsoid given by
\eqref{eq:solidellipsoid} but a perturbed one. In order to mesh it, we need
to find first a parallelepiped that would contain the whole 
ellipsoid perturbed by these higher order terms.  This is done by
defining the following maximal intervals for the 4 components of $\hat s$, 
\[
   \begin{split}
   s_1, \, s_2  \in I_{\epsilon} &:=
               \left [ - \sqrt{\frac{h-h_0}{2\omega_p}} -
   \epsilon, \sqrt{\frac{h-h_0}{2\omega_p}} + \epsilon \right ] \\
   s_3,\, s_4  \in J_{\epsilon} &:=
               \left [ - \sqrt{\frac{h-h_0}{2\omega_v}} -
   \epsilon, \sqrt{\frac{h-h_0}{2\omega_v}} + \epsilon \right ],
   \end{split}
\]
where the value of $\epsilon$ is heuristically chosen in order to account for
the terms $\mathcal
O_3$.

Therefore, we take equally spaced points over $I_\epsilon$ and $J_\epsilon$
for $s_1$, $s_2$ and $s_4$ and we try to solve, for $s_3$ inside the interval
$J_\epsilon$, the equation $H(\Wt(\hat s, 0))=h$ using the whole
expansion for $\Wt$.  In the case that this equation has no solution, then
the point defined by the $(s_1,s_2,s_4)$ is outside the perturbed ellipsoid.
We consider that the equation has no solution when a root for $s_3$ cannot be
numerically bracketed by dyadic subdivisions of the interval $J_\epsilon$ up to a maximum depth. 

This first strategy to mesh $S_j^h$, namely to take $(s_1,s_2,s_4)$ equally
spaced in $I_\epsilon \times I_\epsilon \times J_\epsilon$ and solve for $s_3$
as described, works satisfactorily but admits several computational
optimizations.  One of them, that will be called ``second strategy'' in order
to compare, is not to use the whole expansion of $\Wt$ when bracketing, but
only a truncation.  Several tests have shown that using order 4 for $\Wt$
when bracketing is enough. A third computational improvement (``third
strategy'') is to skip the terms with $s_5=0$ when evaluating the series for
$\Wt$. %
Our fourth and final strategy is
focused on reducing the number of failed attempts to solve for $s_3$ by
changing
the meshing strategy (but keeping the second and third improvements).
Instead of trying all the points of a bounding parallelepiped,
the iso-energetic slice
of the center manifold is assumed to be convex and the mesh
is generated from inside to outside in the three directions $s_1$, $s_2$,
$s_4$. The first time that, following any direction, a solution for $s_3$
cannot be found, the border of the perturbed ellipsoid is considered to have
been reached and this direction is no longer checked.

In order to visualize how these changes improve the efficiency of the
algorithm, a grid of iso-energetic points in $S_1^h$ for $h=-1.5860$ has been
computed. Fixing $\epsilon=0.05$, we take 25 points for each $s_1$, $s_2$ and
$s_4$ inside intervals $I_{\epsilon}$ and $J_\epsilon$ to finally obtain a 
mesh of $S^1_h$ made of 12155 points. In Table \ref{tab:comptimes} the computing
times for
each strategy are presented. As it can be seen, the improvement is quite
drastic and strategy $4$ is the one that presents a better calculation time,
so from now on all the computations of the center-stable and center-unstable
manifolds will be done by using this strategy.

\begin{table}[tbp]
   \centering
   \begin{tabular}{c|c|c|c}
      strategy 1 & strategy 2 & strategy 3 & strategy 4 \\ \hline
      \texttt{86819} & \texttt{526.36} & \texttt{164.15} &
      \texttt{94.409}
   \end{tabular}
   \caption{For the different strategies, computing user time (in seconds) of
   an iso-energetic slice of $W^{c}(L_1)$ on an Intel(R) Xeon(R)
   CPU E5-2630 v3 @ 2.40GHz}
   \label{tab:comptimes}
\end{table}

Our next goal is to determine the domain of validity of the computed expansions
for both center-stable and center-unstable manifolds in terms of the order for
the approximations. Following \cite{mamotreto}, we use the error in the orbit:
given an initial condition on the center-stable or center-unstable manifold
$s_0=(\hat s, \delta), \hat s \in S_j^h$ that is described by the
parameterization $\Wt^{(j)}(s)$, the error in the orbit is given by
\begin{equation} 
   e _O(T,s_0) = \sup_{t\in[0,T]} \norm{\Wt^{(j)}(\phi_t(s_0)) -
   \Phi_t(\Wt^{(j)}(s_0))}, 
\end{equation} 
for a fixed amount of time $T$, that could be negative meaning integration
backwards in time.

Several error experiments have led us to choose
$|\delta|=10^{-3}$. The sign
of $\delta$ is chosen in order to follow the branch of the center-unstable or
center-stable manifold corresponding to the connections we look for, this is,
the branch of $\W^{cu}(L_2)$ that goes towards $L_1$ and the branch of
$\W^{cs}(L_1)$ that comes from $L_2$.  As sketched in Figure~\ref{fig:veps},
the eigenvectors $w_+$ and $w_-$ can be chosen with positive $x$ components
and opposite $y$ components for both libration points. Then, the choice of the
sign of $\delta$ turns out to be negative for the case of $\W^{cs}(L_1)$ and
positive for $\W^{cu}(L_2)$.  In addition, the sign of $T$ is chosen so that
numerical integration is always towards the center manifold. Finally,
in order to use a value of $T$ with physical meaning, it is chosen to be
$|T|=3$, which is close to the periods of the orbits of the planar and vertical
Lyapunov families of p.o.

\begin{figure}[tbp]
   \centering
   \includegraphics[width=.6\textwidth]{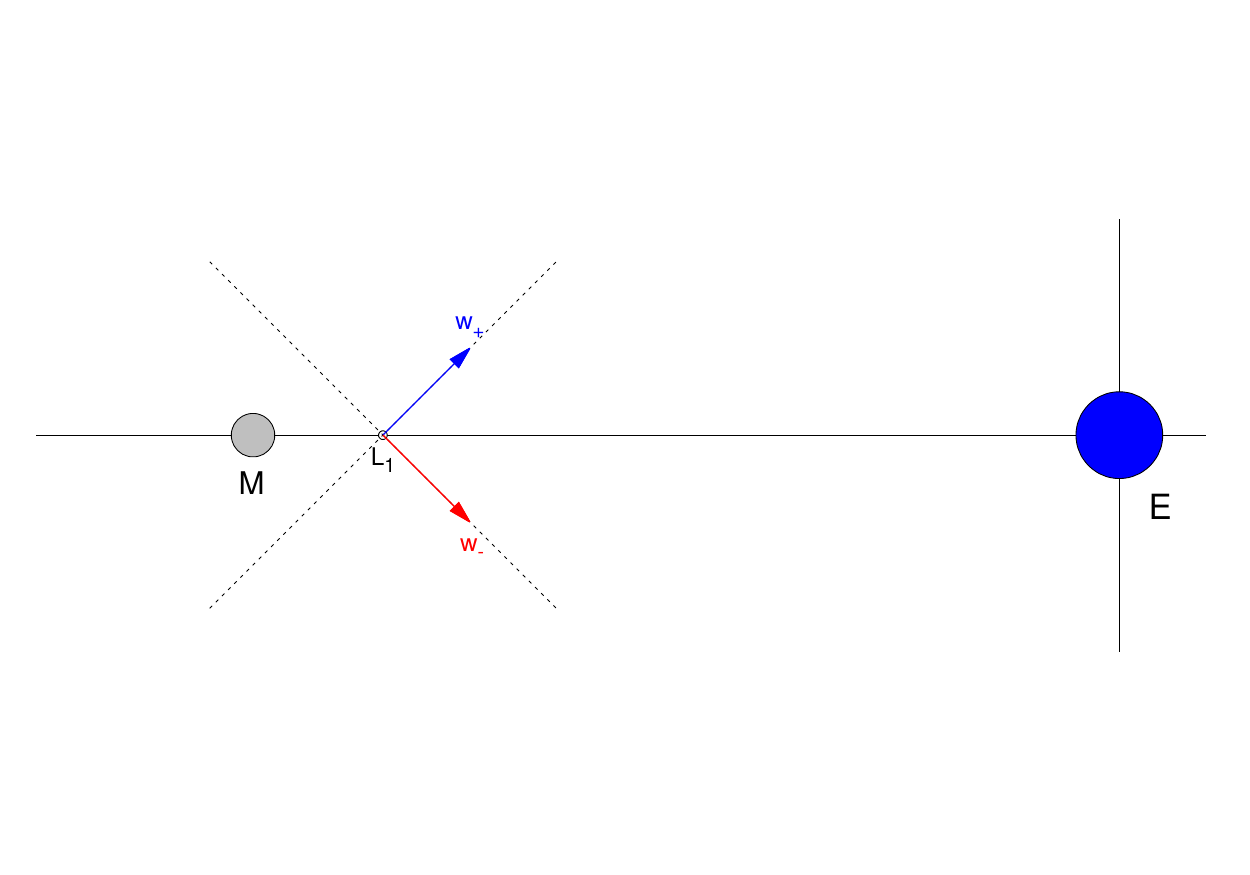}
   \caption{\label{fig:veps}
      Convention of the eigenvectors of the hyperbolic part of $L_1$.
      Analogously for $L_2$.
}
\end{figure}

\begin{figure}[tbp]
   \centering
   \includegraphics{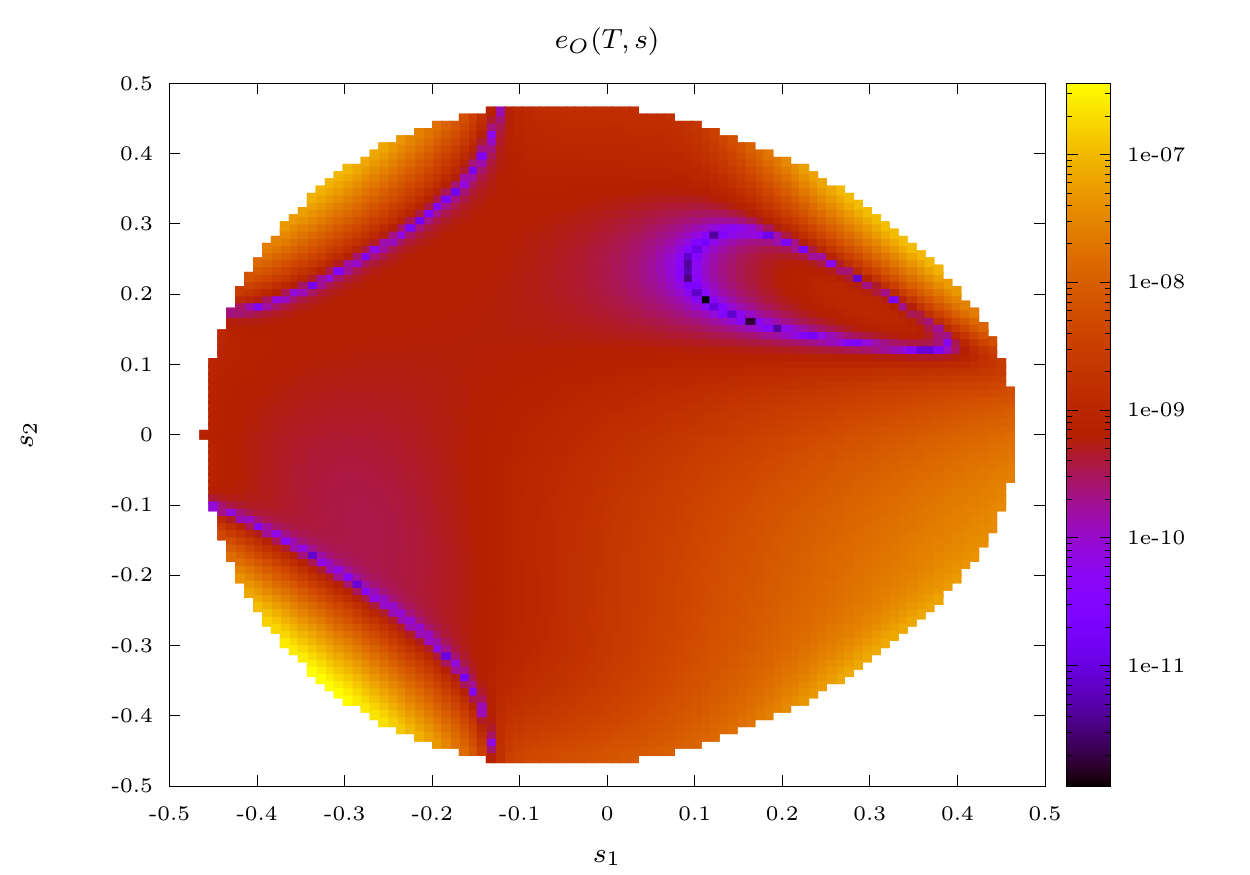}
   \caption{\label{fig:errors}For the expansion of $\W^{cs}(L_1)$ computed up
   to order 20, evaluation of the error in the orbit estimate $e_O$ for
   all points $\hat s$ in the mesh of $S_1^h$ satisfying $s_4=0$, $h=-1.5852$.}
\end{figure}

As a first example, a set of points in $S_1^h$ for $h=-1.586$ is computed
using a parameterization of order 20 of $\W^{cs}(L_1)$.  For each point $(\hat
s, \delta)$ with $\hat s \in S^h_1$, the corresponding orbit error is
computed.  Figure~\ref{fig:errors} represents these errors as a color gradient
map.  For clarity, Figure~\ref{fig:errors} only shows the error for the points
in the mesh of $S_1^h$ with $s_4=0$. With just these points it can be seen
already that the
error is quite heterogeneous in the energy level.

This procedure can be repeated for several energies and several orders of the
expansions in order to determine which order we need to choose when
computing the parameterizations for each energy level. 
In order to consider several energy levels, we take the
maximum error of all the points of the mesh of an iso-energetic slice
and 
plot the error as a function of energy. This is done in
Figure~\ref{fig:errbyorder} left, for $\W^{cs}(L_1)$, and
Figure~\ref{fig:errbyorder} right, for $\W^{cu}(L_2)$. 
Figure~\ref{fig:errbyorder} shows how the error decreases with the order of the
expansions, but at each order increases with energy.
By using order 30, these figures show that we can reach a value of $h=-1.575$
in both libration points with error bounded by
$10^{-4}$.
So, from now on, all the computations will be done using
expansions up to order 30.

\begin{figure}[tbp]
   \centering
   \includegraphics{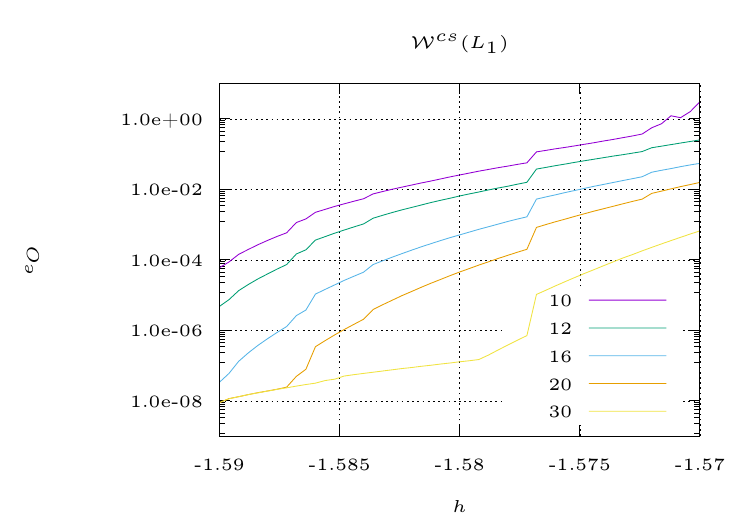}
   \includegraphics{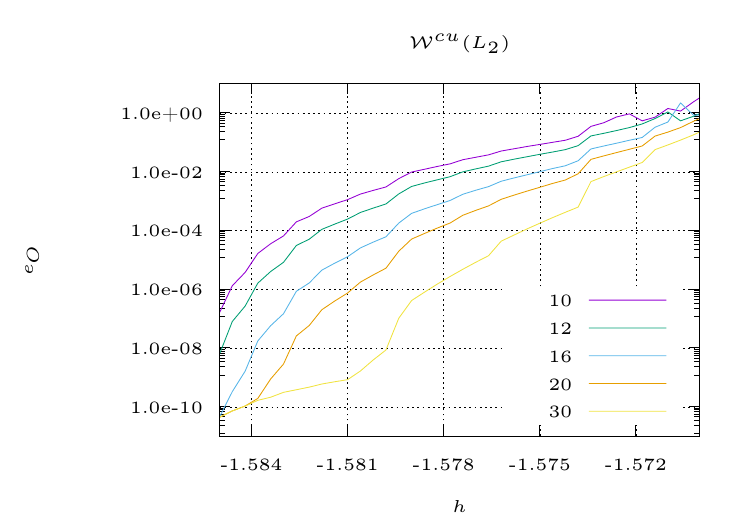}
   \caption{\label{fig:errbyorder}For the expansions of $\W^{cs}(L_1)$
   and $\W^{cu}(L_2)$ computed up to indicated orders, evaluation of the
   maximum error estimate $e_O$ for $|\delta| = 10^{-3}$, $|T|=3$ and different
   energy levels.}
\end{figure}

\section{Computing heteroclinic connections} \label{sec:conexs}

This section is devoted to introduce the main methodology to compute
heteroclinic connections between center manifolds once the iso-energetic slices
of center manifolds have been obtained. For this
purpose, and as it was indicated in Section~\ref{sec:param}, the
invariant manifolds $\W^{cs}(L_1)$ and $\W^{cu}(L_2)$ are considered and, at the
same time, the meshes presented in Section~\ref{sec:mesh} will be used as
starting guesses.

Consider two (real) parameterizations for $\W^{cs}(L_1)$ and
$\W^{cu}(L_2)$, computed according to Section~\ref{sec:param}, that will
be denoted by $\Wt^{(1)}(s)$ and $\Wt^{(2)}(s)$, for $s$ in a neighborhood
of the origin in $\R^5$. From the error explorations of the previous Section, 
for $\Wt^{(1)}(s)$ and $\Wt^{(2)}(s)$ to be accurate we need that
$|s_5|\leq\delta:=10^{-3}$. Therefore, in order to obtain heteroclinic
connections, and also according to the discussion on manifold branches of the
previous Section, $\W^{cs}(L_1)$ (resp.~$\W^{cu}(L_2)$) needs to be extended
through numerical integration, starting from points $\Wt^{(1)}(s)$ with
$s_5=\delta^s:=-\delta$ (resp.~$\Wt^{(2)}(s)$ with $s_5=\delta^u:=\delta$).

Following \cite{2000GoMa,2005GoMaMo,2006CaMa,2007aArMa}, we look for
heteroclinic connections by trying to find intersections of $\W^{cs}(L_1)$ and
$\W^{cu}(L_2)$ in an iso-energetic Poincaré section of energy $h$. This
Poincaré section is given by an hypersurface $\Sigma:=\{g(Z) = 0\}$ for $g
\colon\R^6 \longrightarrow \R$ that should be known to be intersected by
$\W^{cs}(L_1)$ and $\W^{cu}(L_2)$ at the energy level $h$. Since the
connections we are looking for are the ones that join the vicinities of $L_1$
and $L_2$, in our case it will be convenient to take $g(Z) = x - \mu +
1$, for which $\Sigma$ is
perpendicular\index{perpendicular} to the $x$ axis and contains the the second
primary of the RTBP.  We will denote as $P^{\pm k}_\Sigma$ a Poincaré map associated to
$\Sigma$, with $k$ being the number of crossings with the section and
the sign indicating whether integration is done
forward or backward in time. Recalling the $S^h_j$ sets defined in
\eqref{eq:isoenslice}, heteroclinic connections are found by solving the
following equation: find $\hat s^u \in S_2^{h}, \hat s^s \in S_1^{h}$, such
that
\begin{equation}\label{eq:poinceq}
   P_\Sigma^{-k}(\Wt^{(1)}(\hat s^s, \delta^s)) -
   P_\Sigma^{+l}(\Wt^{(2)}(\hat s^u, \delta^u)) = 0.
\end{equation}
The indexes $k$ and $l$ must be chosen in such a way that $k+l$ is an even
number and any pair satisfying that $k+l=2n$ for some $n \in\N$ will give
identical results.

Any root finding method to solve system \eqref{eq:poinceq} requires a good
initial condition. In order to find such initial conditions, we
first choose an energy level $h$ for which heteroclinic connections are
expected to be found (see the next Section). For this energy level, we apply
the meshing procedure of the previous Section in order to obtain points $\hat
S^h_1\subset S^h_1$, $\hat S^h_2\subset S^h_2$. Then, the following sets are
obtained:
\[
   \{P_\Sigma^{+l}(\Wt^{(2)}(\hat s,
   \delta^u))\}_{\hat s \in \hat S^h_2},
   \quad 
   \{P_\Sigma^{-k}(\Wt^{(1)}(\hat s,
   \delta^s))\}_{\hat s \in \hat S^h_1}.
\]
For each point on these sets we compute its Euclidean distance to every
point of the other set and we store the minimum distance and the 
point from the other manifold that gives it.
Then, for each $\hat s^s \in
\hat S_1^h$, 
we obtain the triple $(\hat s^s, \hat s^u(\hat s^s), d(\hat s^s))$ where 
\begin{equation} \label{eq:distance}
   \begin{split}
   d(\hat s^s) &:= \min_{\hat s^u \in \hat S_2^h} 
   \text{dist}(P_\Sigma^{-k}(\Wt^{(1)}(\hat s^s, \delta^s)),
   P_\Sigma^{+l}(\Wt^{(2)}(\hat s^u, \delta^u))), \\
   \hat s^u(\hat s^s) &:= \arg \min_{\hat s^u \in \hat S^h_2} 
   \text{dist}(P_\Sigma^{-k}(\Wt^{(1)}(\hat s^s, \delta^s)),
   P_\Sigma^{+l}(\Wt^{(2)}(\hat s^u, \delta^u))),
   \end{split}
\end{equation}
this is, $\hat s^u(\hat s^s)$ is the value of $\hat s^u$ that gives the minimum in
the expression of $d(\hat s^s)$.
Note that the computation could also be done in terms of $\hat s^u \in
\hat S_2^h$. By defining a
numerical tolerance $\xi$, we can take pairs $(\hat s^s, \hat s^u(\hat s^s))$
such that $d(\hat s^s) \leq \xi$ as initial approximations to solve equation
\eqref{eq:poinceq}. 

The system of equations we actually solve in order to find heteroclinic
connections is not \eqref{eq:poinceq} but a modification of it.
In order to avoid the
computation of iterated Poincaré maps and their differentials, we add the
equation of the Poincaré section to the system of equations to solve and we
also introduce as unknowns $T^s, T^u \in \R$, which are the corresponding times of flight
from the points $\Wt^{(1)}(\hat s^s, \delta^s)$ and $\Wt^{(2)}(\hat s^u,
\delta^u)$, respectively, to the section $\Sigma$.  Since $T^s, T^u$ may become
large, and in order to maintain high precision, a multiple-shooting
strategy is
used: we add as unknowns new points $x_0^s,\dots,x_{m^s}^s$ along the
stable branch and additional points $x_0^u, \dots, x_{m^u}^u$ along the
unstable one. With that, and considering the corresponding matching equations
in these variables, the following system of equations equivalent to
\eqref{eq:poinceq} is obtained:
\begin{equation}\label{eq:systhetero}
   \begin{split}
      \Wt^{(2)}(\hat s^u, \delta^u) - x_0^u &= 0,\\
      \Phi_{T^u/m^u}(x_i^u)-x_{i+1}^u &= 0,
               \qquad i=0,\dots,m^u-1,\\
      \Wt^{(1)}(\hat s^s, \delta^s) - x_0^s &= 0,\\
      \Phi_{T^s/m^s}(x_i^s)-x_{i+1}^s &= 0,
               \qquad i=0,\dots, m^s-1,\\
      g(x_{m^u}^u) &= 0,\\
      H(x_{m^u}^u) - h &= 0,\\
      x_{m^u}^u - x_{m^s}^s &= 0,
   \end{split}
\end{equation}
for the set of unknowns
\begin{equation}\label{eq:unkn}
      \hat s^u,T^u,x_0^u,\dots,x_{m^u}^u,
      \hat s^s,T^s,x_0^s,\dots,x_{m^s}^s.
\end{equation}

In spite of being apparently more cumbersome, a solver for this system is
actually simpler to code than one for \eqref{eq:poinceq}, since
Poincaré maps are no longer explicitly present, and the introduction of multiple shooting
points prevents the appearance of the parameterization composed in the flow.
The differential of
the flow is computed through numerical integration of the first variational
equations.

In this system, the number of unknowns is $n(m^u+m^s+2)+2d$, while the
number of equations is $n(m^u+m^s+3)+2$ so, for any
values of $m^u$, $m^s$, we will have $2d-n-2$ more unknowns than equations.
Particularly, for $n=6$ and $d=5$, this is 2 more unknowns than equations.
Since, by a dimensional argument (and as numerically found in a similar
problem \cite{2007aArMa}), the set of heteroclinic connections in an energy
level is expected to be two-dimensional, having 2 more unknowns than equations
is coherent with this fact, and system \eqref{eq:systhetero} should have full
rank. Numerical checks with SVD (see e.g.~\cite{golub-vanloan}) have
confirmed this last fact.

System \eqref{eq:systhetero} has been solved by performing minimum-norm
lest-squares Newton corrections, that can be computed through QR
decompositions with column pivoting (\cite{2001GoMo,2019aMo}) as long as the
dimensions of the kernel is known. This methodology addresses the fact that
number of equations and unknowns is different. Moreover, from
\eqref{eq:systhetero}, several other systems of interest can be obtained by
eliminating equations and/or fixing values of unknowns. For instance, one can
fix $s_3^s= s_4^s=s_3^u=s_4^u=0$
in order to look for heteroclinic connections between planar Lyapunov p.o.,
thanks to the parameterization style we have used (see Section~\ref{sec:param}).
All of these derived systems are easily coped with by the same routine.

\section{Numerical results}\label{sec:numres}

In this section we apply all the developments of
Sections~\ref{sec:param},~\ref{sec:mesh},~\ref{sec:conexs} to the computation
of whole sets of heteroclinic connections between center manifolds of $L_1$,
$L_2$ of the spatial, circular RTBP in the Earth-Moon case for several
energy levels. All the numerical integration has been performed through
a Runge-Kutta-Felbergh method of orders 7 and 8 with relative
tolerance $10^{-14}$. System \eqref{eq:systhetero} has been solved
with absolute tolerance $10^{-10}$.  All the explorations presented in this
section have been carried out on a Fujitsu Celsius R940 workstation, with two
8-core Intel Xeon E5-2630v3 processors at 2.40GHz, running Debian GNU/Linux 11
with the Xfce 4.16 desktop. The source code has been written in C, compiled
with GCC 10.2.1 and linked against LAPACK 3.7.0. The code also
uses OpenMP $4.0$ in order to parallelize most of the
computations. All the computational times presented in this Section refer to
total computing (user) time, accounting for all the cores. In the
workstation mentioned, wall-clock time is roughly user time divided by 16.

As mentioned in Section~\ref{sec:conexs}, system \eqref{eq:systhetero} is
easily modified in order to look for heteroclinic connections of planar
Lyapunov orbits. When such connections exist, heteroclinic connections between
tori are expected to be found nearby. As done in previous works
\cite{2005GoMaMo,2006CaMa}, we start by looking for connections
between planar Lyapunov
p.o.~by fixing $k=1$ for a set of different energy levels and varying $j$
until the manifold tubes of the departing and arriving p.o.~of the energy
level seem to intersect at $\Sigma=\{x=\mu-1\}$.  Since neither the planar
Lyapunov p.o.~nor their stable and unstable manifolds have $z,p_z$ components
(this is, they are planar), this is easily visualized in an $xy$ projection. A
$yp_y$ projection of the Poincaré sections of the manifold tubes can confirm
if the manifold tubes intersect or not\footnote{An intersection in the $yp_y$
projection of the Poincaré section is a true intersection, since $x=\mu-1$,
$z=p_z=0$ and the manifold tubes have the same energy.}.

Figure~\ref{fig:lyaps} is a sample of results related to the search of
heteroclinic connections of planar Lyapunov p.o. The left column displays $xy$
projections of the manifold tubes, whereas the right column displays $yp_y$
projections of the Poincaré sections of the manifold tubes. The energies have
been chosen in such a way that $2$ heteroclinic connections of p.o.~exist. It
can be seen that, as energy increases, the planar Lyapunov p.o.~around $L_2$
becomes bigger. Its
unstable manifold tube also becomes bigger, in such a way that, with a
smaller value of $j$, this tube is able to intersect the stable manifold tube
of the planar Lyapunov p.o.~around $L_1$. This happens when going from the
first to the second row of Figure~\ref{fig:lyaps}, and from the second to the
third. Observe also that some sections of the manifold tubes are not
represented as closed curves in the $yp_y$ projections of the right column of
Figure~\ref{fig:lyaps}.  This is because of their proximity to the Moon: we have
stopped numerical integration at two radii distance of its center. This can be
also appreciated in the $xy$ plots of the left column of Figure~\ref{fig:lyaps}.

\begin{figure}[htbp]
   \centering
   \includegraphics{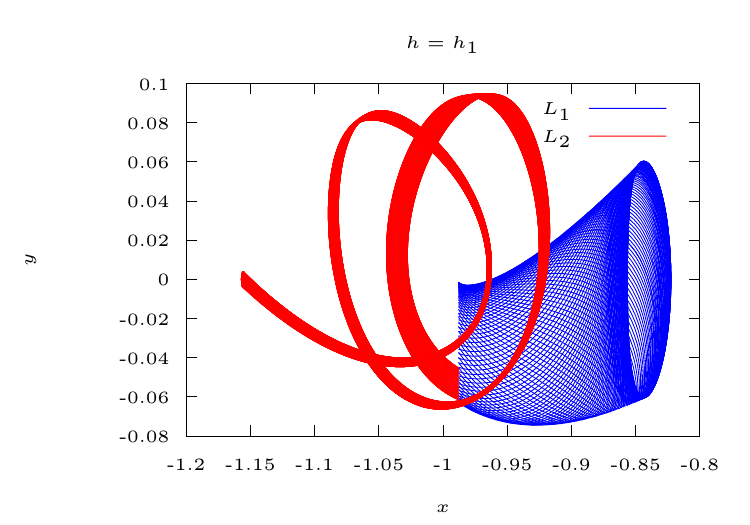}
   \includegraphics{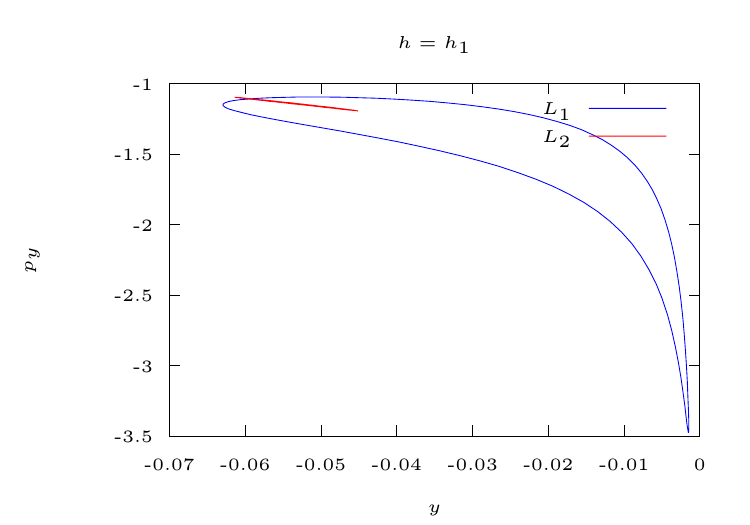}
   \includegraphics{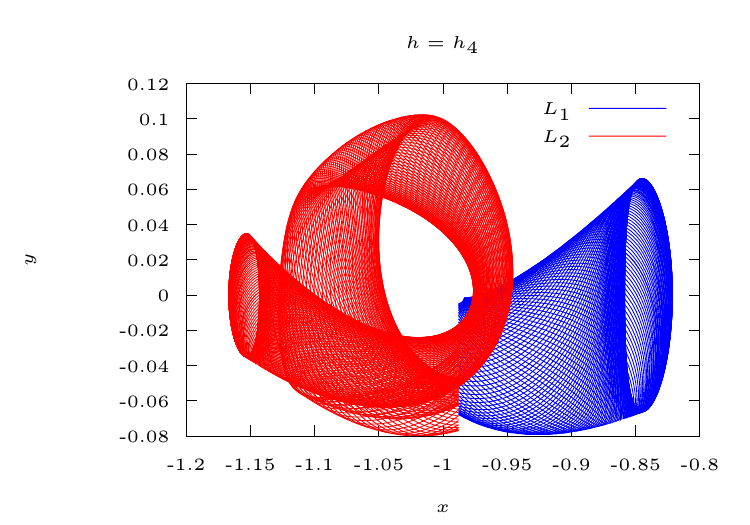}
   \includegraphics{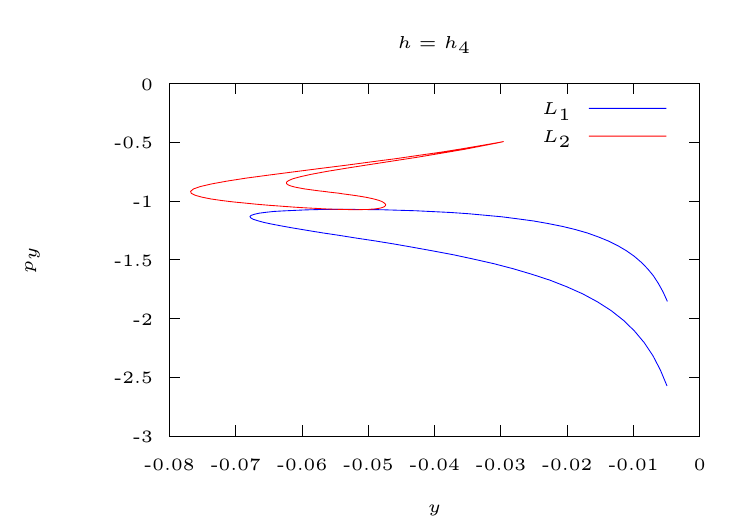}
   \includegraphics{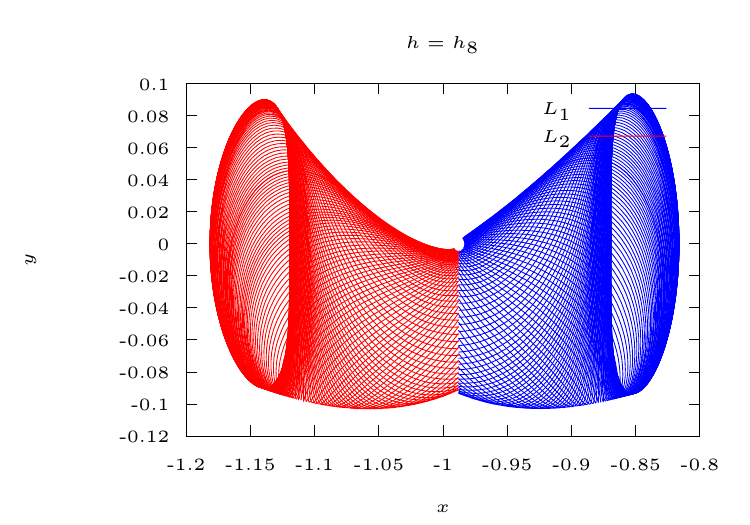}
   \includegraphics{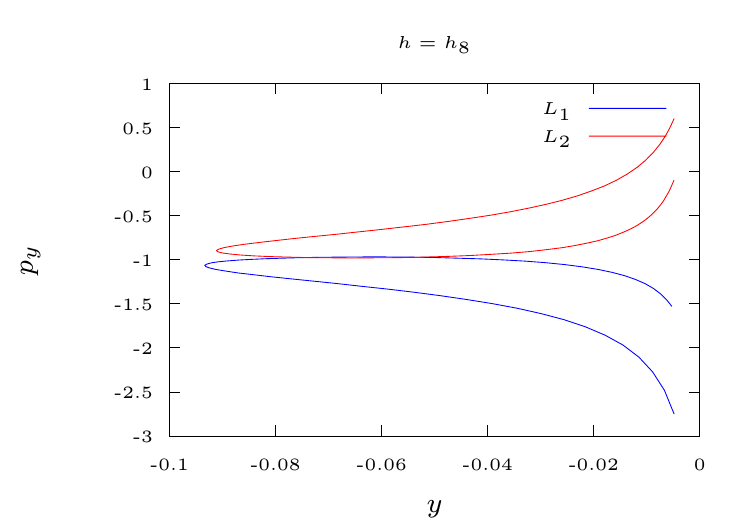}
   \caption{\label{fig:lyaps} 
   Numerical explorations of the unstable and stable branches of the planar
   Lyapunov p.o.~for the energy levels $h_1$, $h_4$
   and $h_8$
   and $j=1,3,5$ crossings.  Left:
   $xy$ view of the
   2-dimensional manifold tubes up to the Poincaré section $\Sigma$.
   Right: Poincaré section of the 2-dimensional manifold tubes
    in the $yp_y$ plane projection. The values
   $\{h_i\}_{i=1}^9$ are given in Table~\ref{tbl:energies}.}
\end{figure}

When looking for heteroclinic connections of planar Lyapunov p.o.~we only
consider primary heteroclinics, in the sense that they correspond to the first
intersection of the manifold tubes. Further intersections are possible, giving
rise to ``higher-order'' heteroclinics. E.g., when in the middle row of
Figure~\ref{fig:lyaps} we locate heteroclinic connections for $j=3$, the
heteroclinic connections for $j=5$ still exist. But they are very tangled,
with one manifold tube performing infinite loops around the other
(see~\cite{2009BaMoOll} for some discussion of this phenomenon). Nearby
heteroclinic connections between tori should inherit this complicated
geometry. They will not be considered here.

In view of Figure~\ref{fig:lyaps}, we set as goal the description of the whole
set of heteroclinic connections from $\W^c(L_2)$ to
$\W^c(L_1)$ for nine energy levels that, in groups of three, correspond to each of
the behaviors of the manifold tubes shown in this figure. The energy levels
chosen are shown in Table~\ref{tbl:energies}, and will be globally denoted as
$\{h_i\}_{i=1}^9$.

\begin{table}[htbp]
   \[
      \begin{array}{c|r@{}c@{}l|c|c}
         i & \multicolumn{3}{c|}{h_i} & j & k  \\
         \hline
         1 & -1&.&58606 & 1 & 5 \\
         2 & -1&.&5855 & 1 & 5  \\
         3 & -1&.&585 & 1 & 5
      \end{array}
      \qquad
      \begin{array}{c|r@{}c@{}l|c|c}
         i & \multicolumn{3}{c|}{h_i} & j & k \\
         \hline
         4 & -1&.&5845 & 1 & 3 \\
         5 & -1&.&5844 & 1 & 3 \\
         6 & -1&.&5843 & 1 & 3
      \end{array}
      \qquad
      \begin{array}{c|r@{}c@{}l|c|c}
         i & \multicolumn{3}{c|}{h_i} & j & k \\
         \hline
         7 & -1&.&5755 & 1 & 1 \\
         8 & -1&.&575 & 1 & 1 \\
         9 & -1&.&574 & 1 & 1 
      \end{array}
   \]
\caption{\label{tbl:energies}Energy levels $\{h_i\}_{i=1}^9$ for which the set
of heteroclinic connections from $\W^c(L_2)$ to
$\W^c(L_1)$ is to be found.}
\end{table}

\begin{figure}[htbp]
   \centering
   \includegraphics{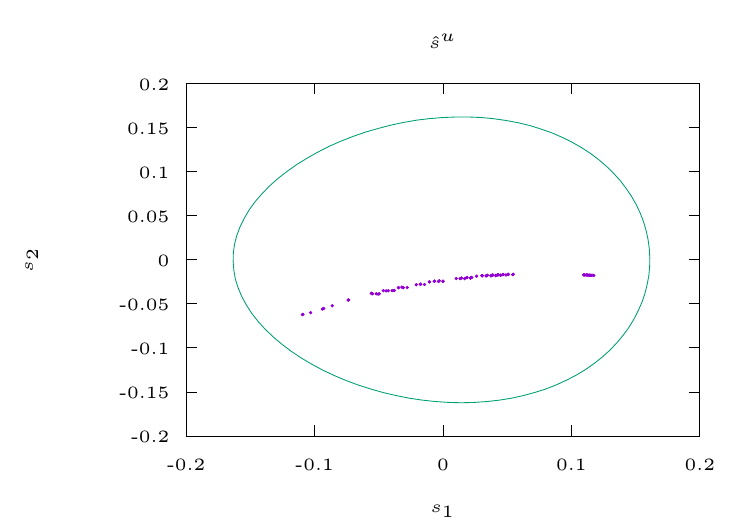}
   \includegraphics{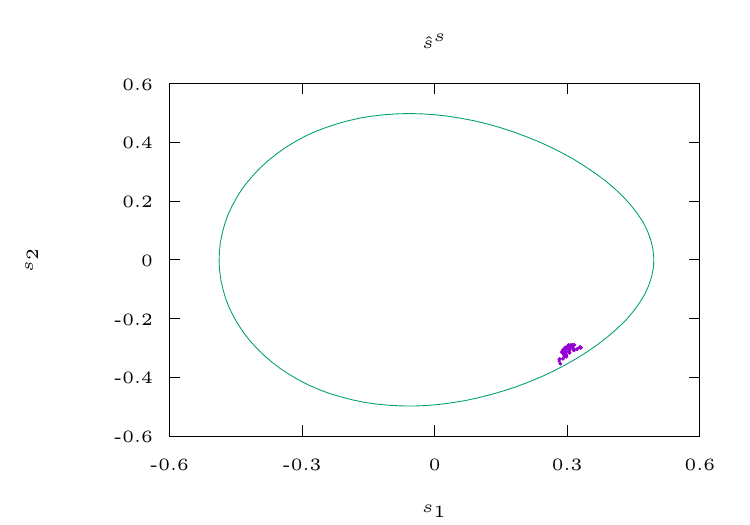}
   \includegraphics{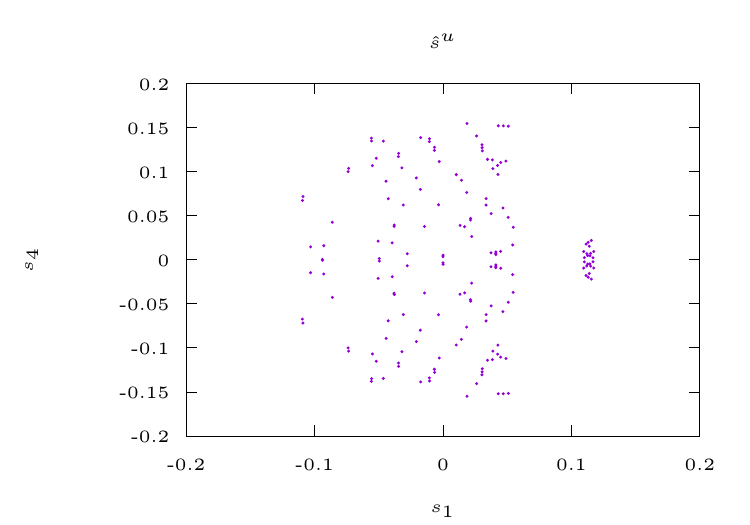}
   \includegraphics{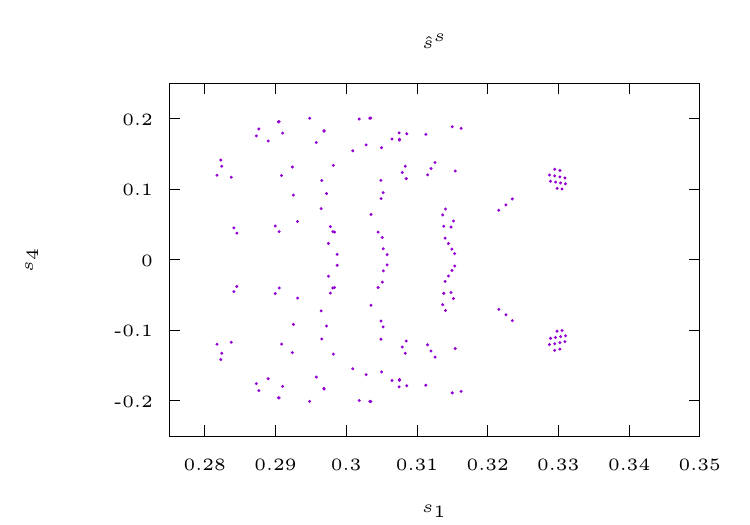}
   \caption{\label{fig:conexionsinc} 
      ``Low-res'' heteroclinic connection set from $\W^c(L_2)$ to $\W^c(L_1)$
      for the energy level $h=h_3$ (purple) 
      in terms of 
      $\hat s^u \in S_2^h$ (left) and $\hat s^s \in S_1^h$ (right) from
      \eqref{eq:unkn}. Green: planar Lyapunov p.o.,
   included as a reference.
}
\end{figure}

In order to give an idea of the steps to follow and the computational effort
required, we provide the details of the computation of the set of heteroclinic
connections for $h=h_3$, with $k=1$ and $j=5$. The corresponding iso-energetic
slices $S^h_1$, $S^h_2$ (see
eq.~\eqref{eq:isoenslice}) are meshed according to Section~\ref{sec:mesh}.
The $\hat S^h_1$ mesh spacing in $s_1,s_2,s_4$ is of $0.0079$
units, with a total of $2049031$ points that are obtained in $16106$ seconds.
The $\hat S^h_2$ mesh spacing is $0.0034$, with a total of
$826599$ points that are
obtained in $6742$ seconds. Setting a tolerance $\xi = 10^{-4}$ gives a set
of pairs $(\hat s^s, \hat s^u)$ as initial approximations of heteroclinic
connections that are later refined using system \eqref{eq:systhetero} as
explained in Section~\ref{sec:conexs}. The set of
heteroclinic connections obtained are presented in
Figure~\ref{fig:conexionsinc} in terms of the values of $\hat s^u,\hat s^s$
found solving system \eqref{eq:systhetero}. This set, made of 168
connections, that have been refined in 68 seconds, provides a
``low-res'' representation of the total heteroclinic connection set.

\begin{figure}[h]
   \centering
   \includegraphics{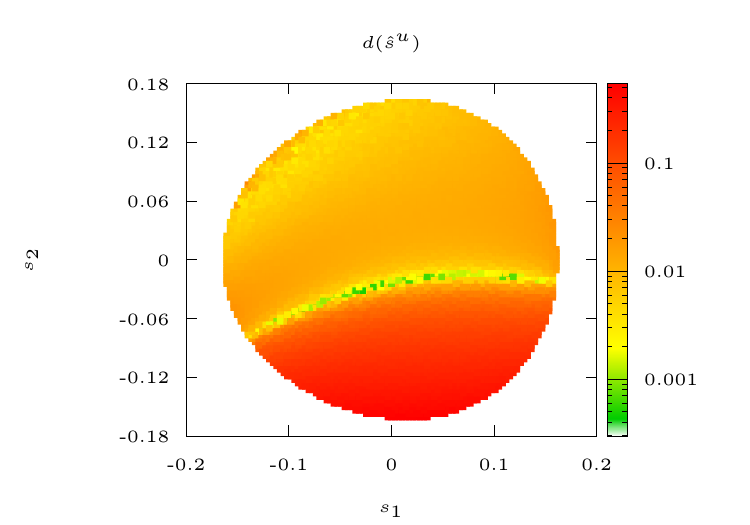}
   \includegraphics{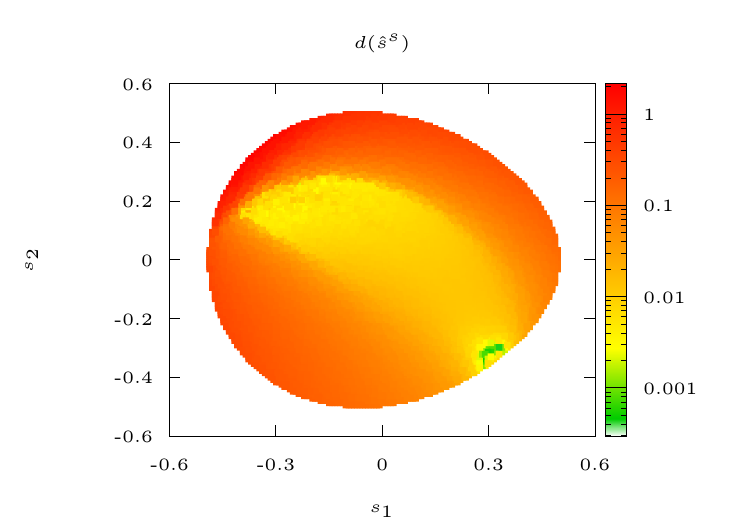}
   \caption{\label{fig:heatmap} 
   Color gradient map provided by
   equation~\eqref{eq:distance} for the
   iso-energetic
   slice of the center manifold at $h=h_3$.
}
\end{figure}

Aiming to refine these results, and in order to describe the whole set
of heteroclinic connections, the values $d(\hat s^s)$
are computed for $\hat s^s$ varying on the mesh of $
S^h_1$ according to the
expression~\eqref{eq:distance}, and the values $d(\hat s^s)$ are computed for
$\hat s^s$ varying on the mesh of $S^h_2$ according to
an analogous
expression. These
values are represented as gradient maps in Figure~\ref{fig:heatmap}. These
figures are useful to determine which regions over each $S_j^h$ need to be
re-meshed with a finer grid. We have used a different re-meshing strategy
depending of the shape of these regions. In the case of
Figure~\ref{fig:heatmap} right, a bounding rectangle in
$s_1,s_2$. In the case
of Figure~\ref{fig:heatmap} left, in which the region seems
to be ``stretched
along a curve'', this curve is approximated by a degree 5 interpolating
polynomial $s_2=p(s_1)$ that is then fattened in $s_2$. In both cases, $s_4$
is allowed to range between the minimum and maximum values attained in the
low-res representation of the set of connections of
Figure~\ref{fig:conexionsinc}.  Once these regions are re-meshed, they are
propagated again up to the Poincaré section, and the triples $(\hat s^s,
\hat s^u( \hat s^s), d(\hat s^s))$ are computed again. From these
triples, and fixing again a tolerance $\xi=10^{-3}$, a new set of
initial approximations to heteroclinic connections is found, which, through the
solution of system \eqref{eq:systhetero}, is refined to a set of
true connections up to tolerance of $10^{-10}$. Thus obtaining 45390
connections in a total of 17628 seconds. The results are
shown in Figure~\ref{fig:resultscomp}. 

\begin{figure}[htbp]
   \centering
   \includegraphics{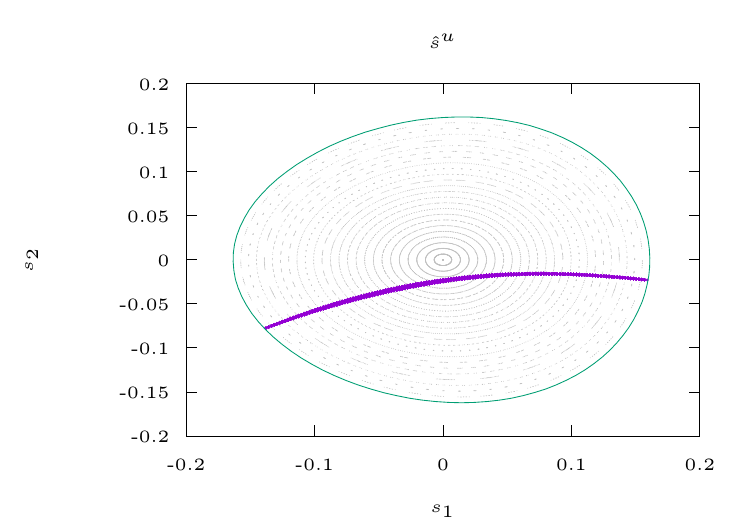}
   \includegraphics{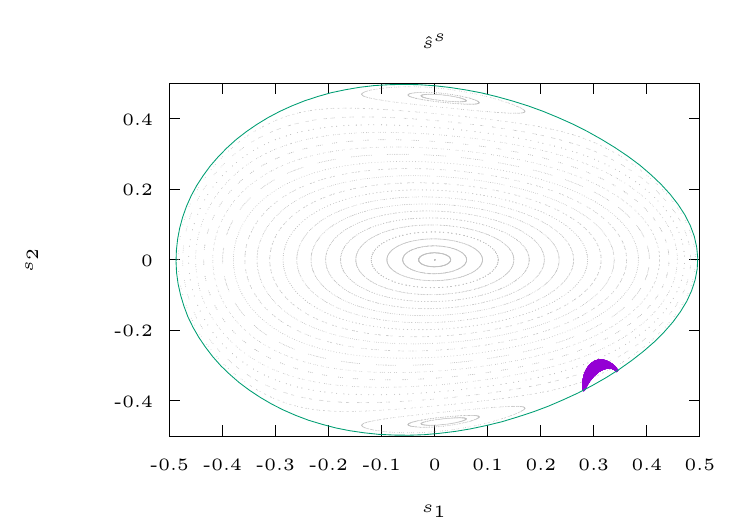}
   \includegraphics{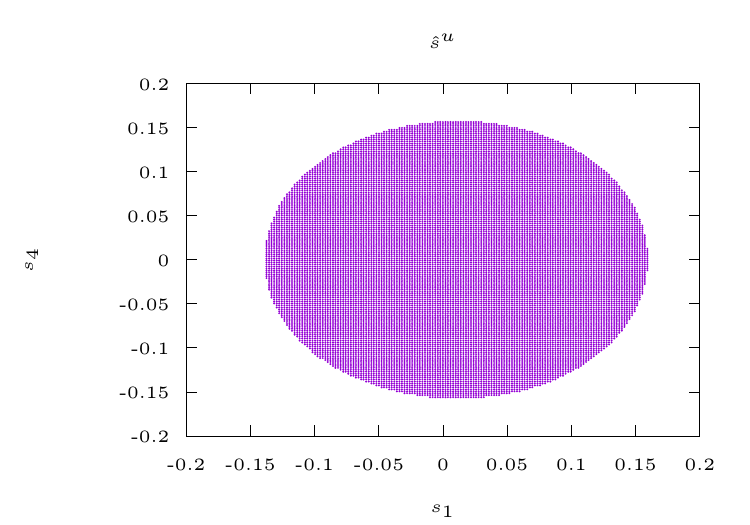}
   \includegraphics{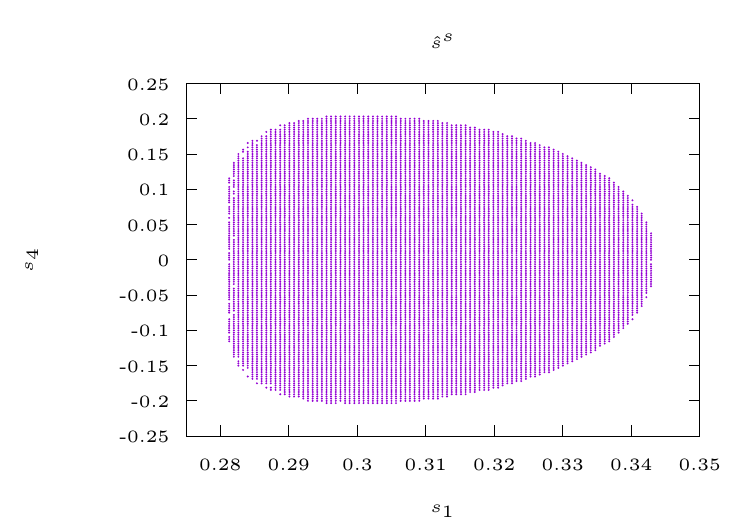}
   \caption{\label{fig:resultscomp}
      Whole set of heteroclinic connections for the energy level $h=h_3$
      from $\W^c(L_2)$ to
      $\W^c(L_1)$ (purple). Green: planar Lyapunov p.o., included as a reference.
      Grey: Poincaré sections of
      several trajectories at $\Sigma:=\{s_4=0\}$.
   }
\end{figure}

Figure~\ref{fig:resultscomp} is a first figure that could be of
special interest for a
space mission analyst. As an addition to
Figure~\ref{fig:conexionsinc}, and
following the classical representations of \cite{esa3,1999JoMa}, the
$\{s_4=0\}$ Poincaré sections of several trajectories of the energy level have
been represented in grey in the first row of
Figure~\ref{fig:resultscomp}. In this
way, Lyapunov and Halo periodic orbits are seen as fixed points ($s_1=s_2=0$
corresponds to the
vertical Lyapunov p.o.~and the two other points correspond to the Halo orbits
of the energy level), and invariant tori are seen as invariant curves (tori
can be found from the Lissajous and quasi-halo families).
From the right column
of Figure~\ref{fig:resultscomp}, it can be seen that, for this energy level, a
small range of the tori close to the planar Lyapunov p.o.~around $L_1$ are
connected to most of the tori around $L_2$. So, assuming this energy level is
convenient, the small range of tori close to the planar p.o.~would be the
place for a space servicing station, that would have ``free routes'' (through
the heteroclinic connections) to most of the tori around $L_2$ (always in the
$h_3$ energy level).  By converting all the points of these figures to synodic
coordinates and using the same scale in both axes, even an idea of the
physical shape of the connected tori could be inferred.

Still considering applications to mission analysis, the last row of
Figure~\ref{fig:resultscomp} would be useful as well. In the last step of its
computation, namely the solution of system \eqref{eq:systhetero}, the
coordinates $s_1,s_4$ in either $\W^{cs}(L_1)$ or $\W^{cu}(L_2)$  can be kept
constant during Newton iterations, and, in this way, obtain an equally spaced
set of points inside the region of heteroclinic connections.
This has actually been done in the second row of
Figure~\ref{fig:resultscomp}.  Over these diagrams, properties of interest of
the connection (such as time of flight, $z$-amplitude, minimum distance to the
Moon) could be graphed.

From Figure~\ref{fig:resultscomp}, it could seem that the set of heteroclinic
connections is a 2-dimensional surface with boundary. This is not the case: in
the re-meshing of the connection candidates, for most of them two solutions
are found for the $s_3$ coordinate. Many survive as heteroclinic connections.
If we use the $s_3$ coordinate in the plot of the set of heteroclinic
connections, a topological sphere is obtained (see
Figure~\ref{fig:h3sphrs}).  This is
coherent with the results in \cite{2007aArMa}. This is also coherent with the
fact that the surfaces of connections displayed in
Figure~\ref{fig:resultscomp} have its apparent boundary on the boundary of the
solid perturbed ellipsoid of the $s_1s_2s_4$
projection of $S_1^h$, $S_2^h$. For $S_2^h$, this fact
can be appreciated in Figure~\ref{fig:h3Elips-l2}, where the mesh obtained for
this perturbed
ellipsoid is represented both in center manifold and synodic coordinates.
The
representation in synodic coordinates of Figure~\ref{fig:h3Elips-l2} uses the
same scale in all axes, and, in this way, provides the actual physical
appearance in configuration space of an iso-energetic slice of the center
manifold.

\begin{figure}[htbp]
   \centering
   \includegraphics{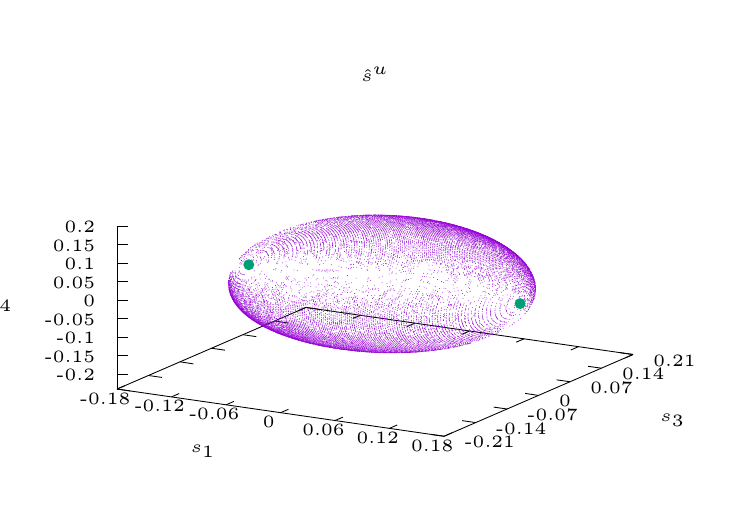}
   \includegraphics{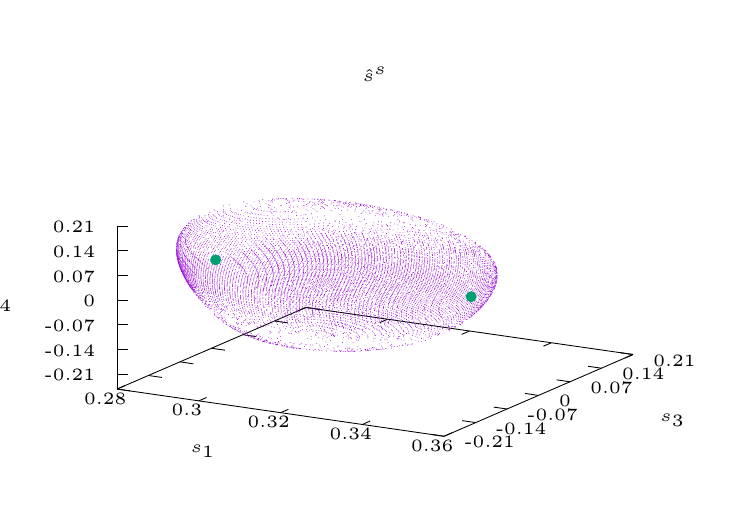}
   \caption{\label{fig:h3sphrs} 
      Whole set of heteroclinic connections from
      $\W^{c}(L_2)$ to
      $\W^{c}(L_1)$ for the energy level $h=h_3$ (purple).
      Green: heteroclinic
      connections between the planar Lyapunov p.o.~at the same energy.
   }
\end{figure}

\begin{figure}[htbp]
   \centering
   \includegraphics{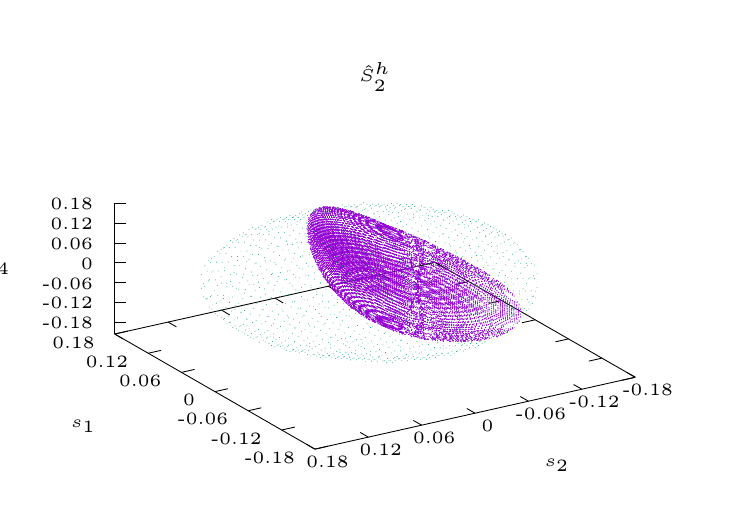}
   \includegraphics{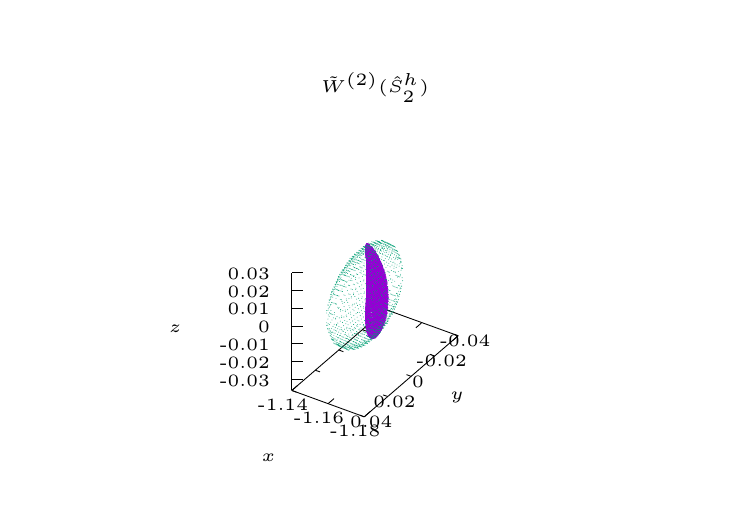}
   \caption{\label{fig:h3Elips-l2} 
      Representation of the boundary of the obtained mesh for the perturbed
      ellipsoid $S_2^h$ at
      $h=h_3$ (green) together with the heteroclinic connections $\hat s^u \in
      S_2^h$ from $\W^c(L_2)$ to $\W^c(L_1)$ (purple). Left: representation in
      the center manifold. Right: representation in synodic coordinates.
}
\end{figure}

In Figure~\ref{fig:esf} we represent the whole set of connections from
$\W^c(L_2)$ to
$\W^c(L_1)$ for the energies $h_3,h_6,h_9$, by
choosing in each case a 3D projection suitable to better see them as
topological spheres. The two
heteroclinic connections between the planar Lyapunov p.o.~of the energy level
are shown as points in the same plot. As an illustration of how individual
connections look like, for the
energies $h_2$, $h_5$, $h_7$ we represent the $xy$ and $yz$ projections of the
connections passing by the moon at minimum distance
(Figure~\ref{fig:dmin}), and also the connections
having maximum $z$-amplitude (Figure~\ref{fig:zmax}). In order to show
the
true physical appearance of these connections, the same scale has been used in
both axes in all the plots. For the same reason, the Moon has been added to the
$xy$ projections.

\begin{figure}[tbp]
   \centering
   \includegraphics{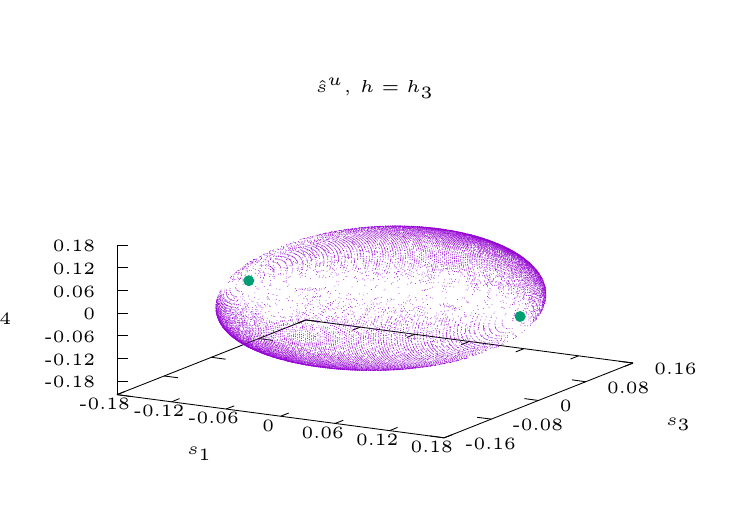}
   \includegraphics{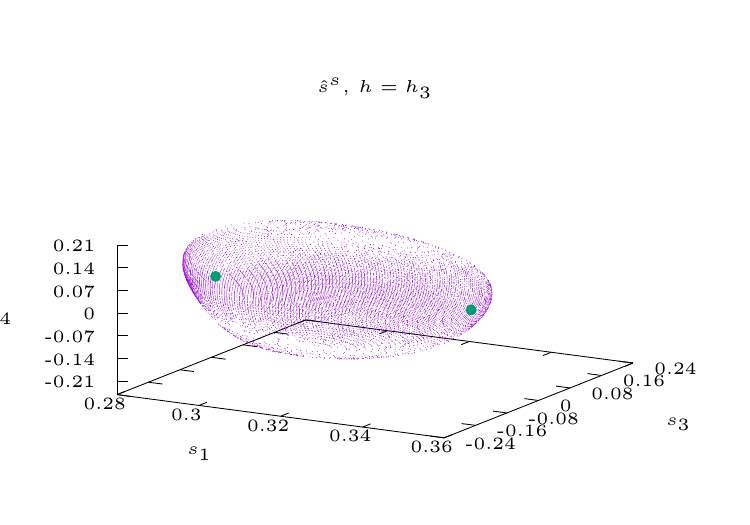}
   \includegraphics{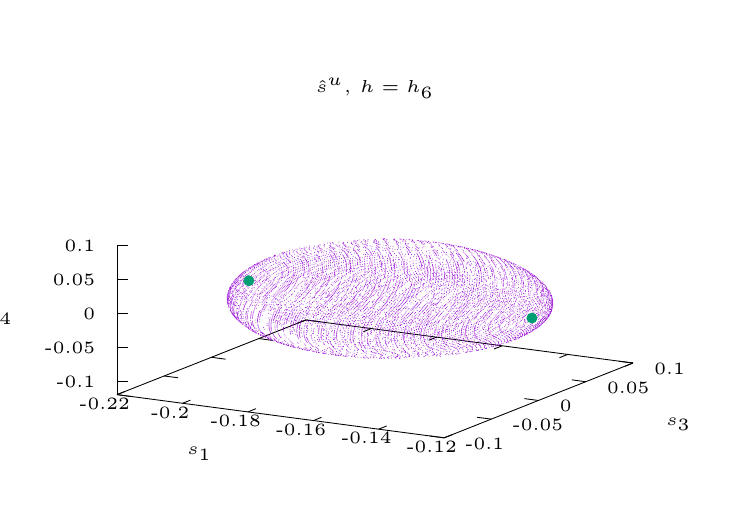}
   \includegraphics{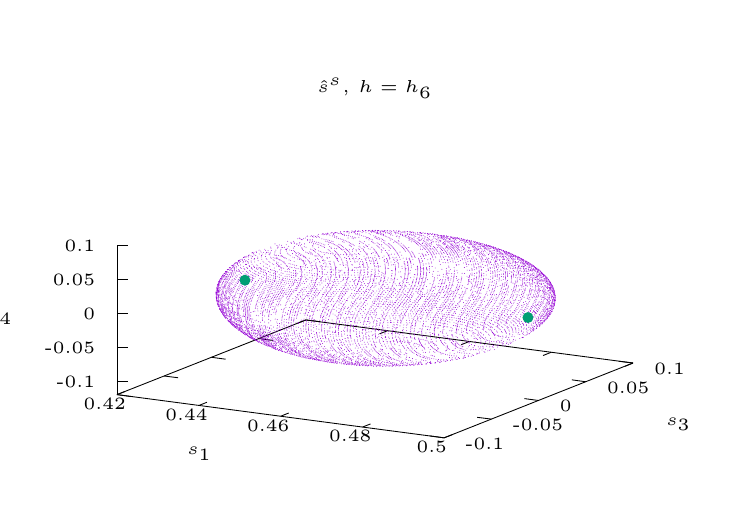}
   \includegraphics{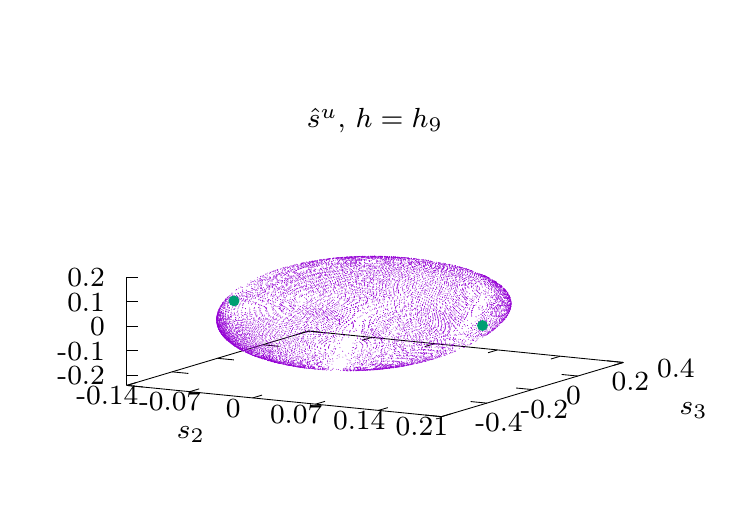}
   \includegraphics{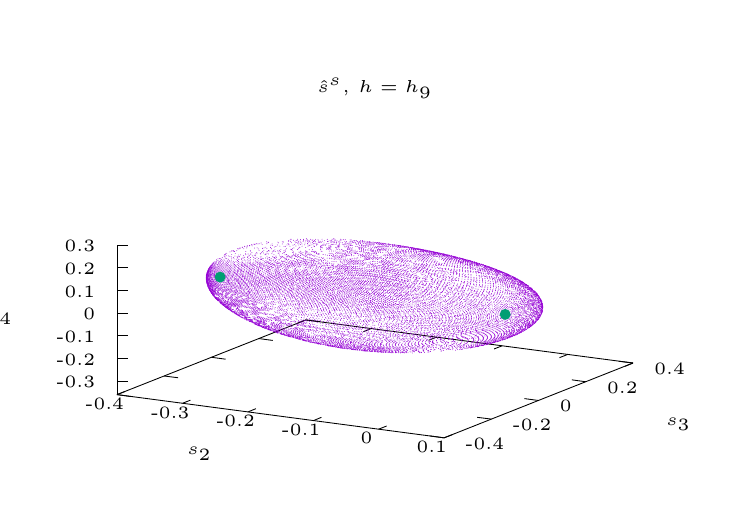}
   \caption{\label{fig:esf}
   Representation of the whole set of heteroclinic connections from
   $\W^{c}(L_2)$ to $\W^{c}(L_1)$ (purple) for several energy levels $h_i$,
   $i=3,6,9$, and in terms of $\hat s^u \in S_2^h$ (left) and $\hat s^s \in S_1^h$
   (right) from \eqref{eq:unkn}. Green: heteroclinic connections between planar
Lyapunov p.o..    }
\end{figure}

\begin{figure}[tbp]
   \centering
   \includegraphics{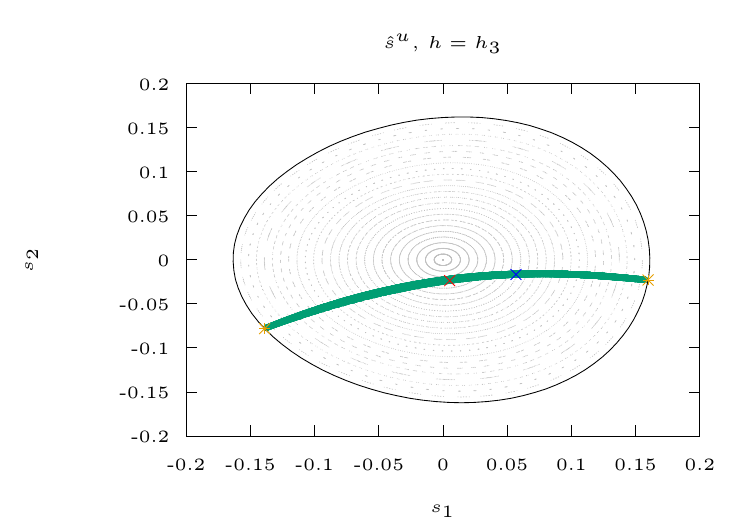}
   \includegraphics{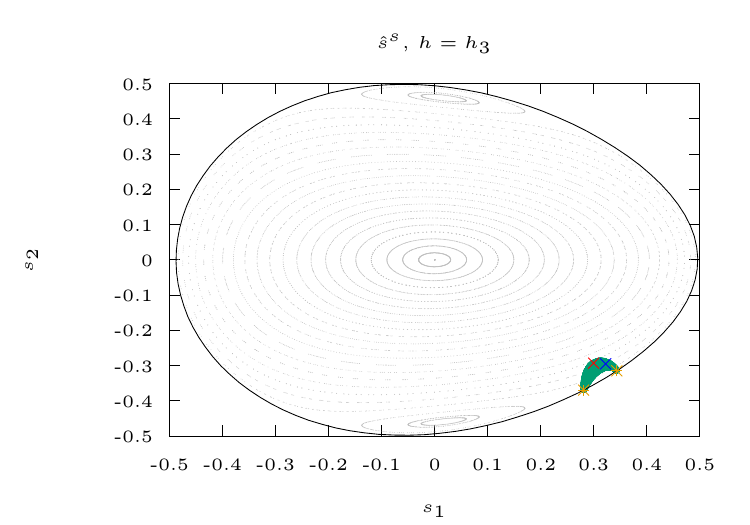}
   \includegraphics{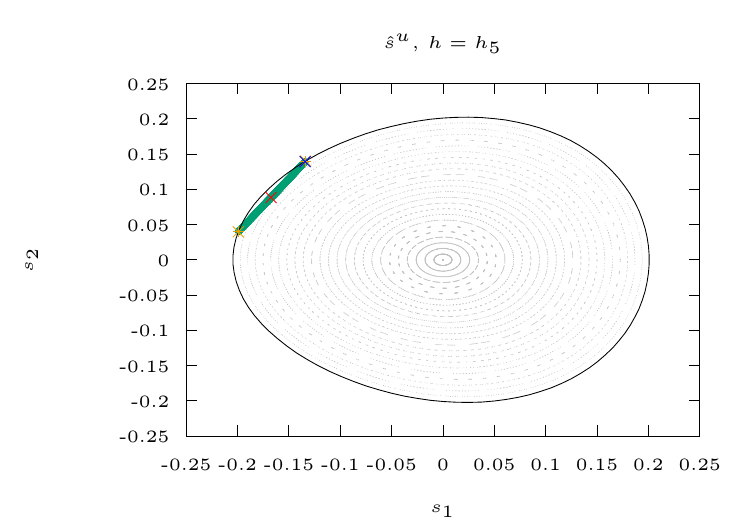}
   \includegraphics{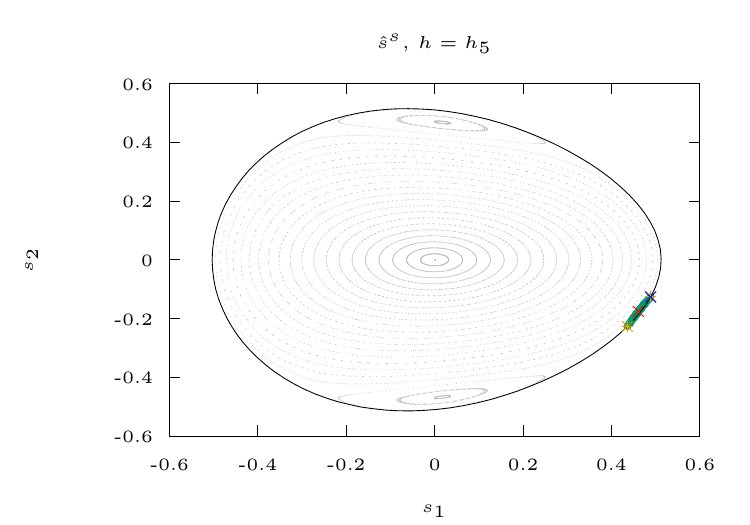}
   \includegraphics{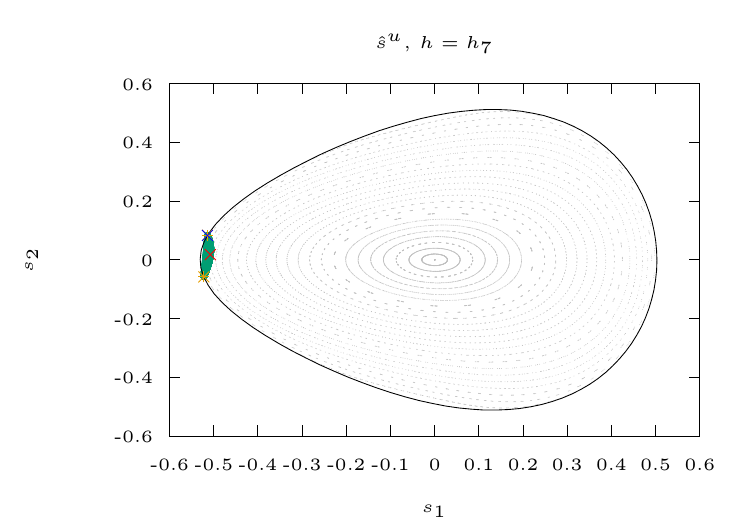}
   \includegraphics{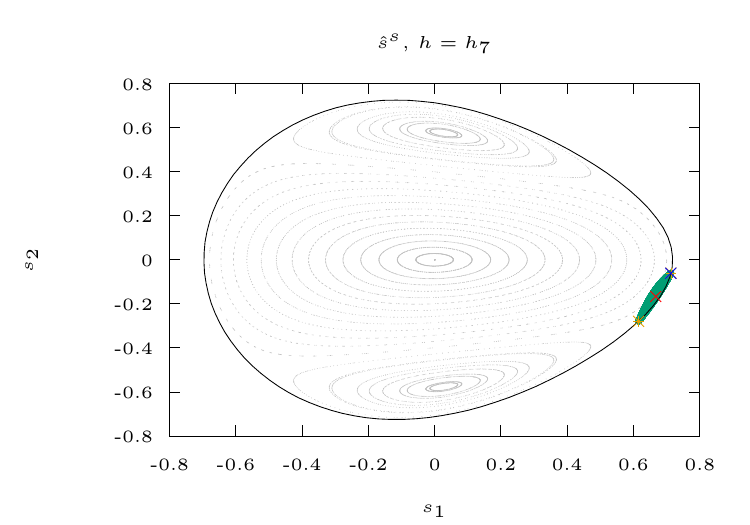}
   \caption{\label{fig:secp} 
      Representation of the whole set of heteroclinic connections from
      $\W^{c}(L_2)$ to $\W^{c}(L_1)$ (green) for the energy levels $h_i$,
      $i=3,5,7$, and in terms of $\hat s^u \in S_2^h$ (left) and $\hat s^s \in
      S_1^h$ (right) from \eqref{eq:unkn}. Yellow: heteroclinic connections
      between planar Lyapunov p.o..  Black: planar Lyapunov
      p.o., included as
      reference.  Grey: Poincaré sections of several trajectories at
      $\Sigma:=\{s_4=0\}$. Blue: initial condition of the heteroclinic
      connection with minimum distance presented in Figure~\ref{fig:dmin}. Red:
      initial condition of the heteroclinic connection with maximum
      $z$-amplitude
      presented in Figure~\ref{fig:zmax}.
}
\end{figure}

\begin{figure}[tbp]
   \centering
   \includegraphics{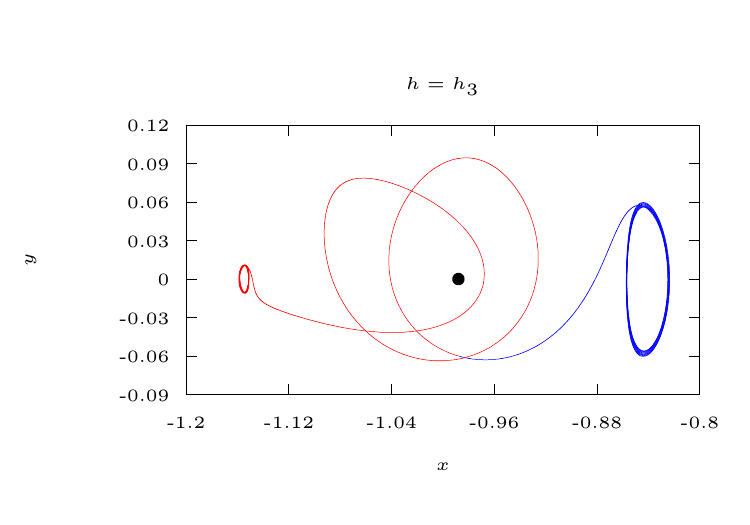}
   \includegraphics{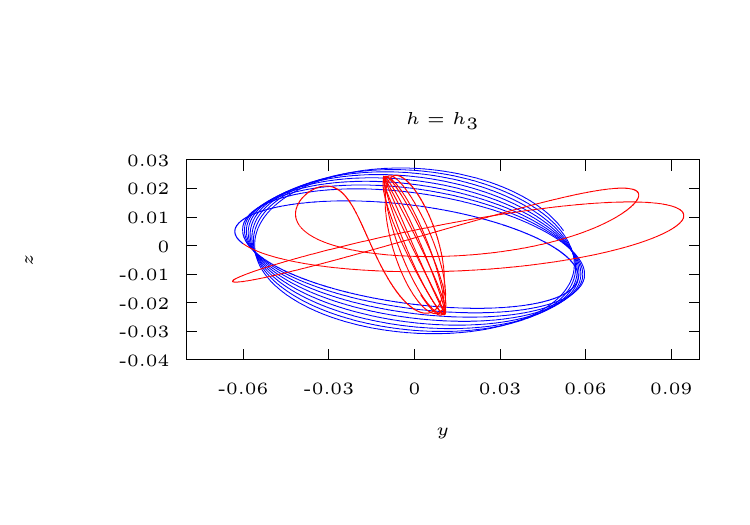}
   \includegraphics{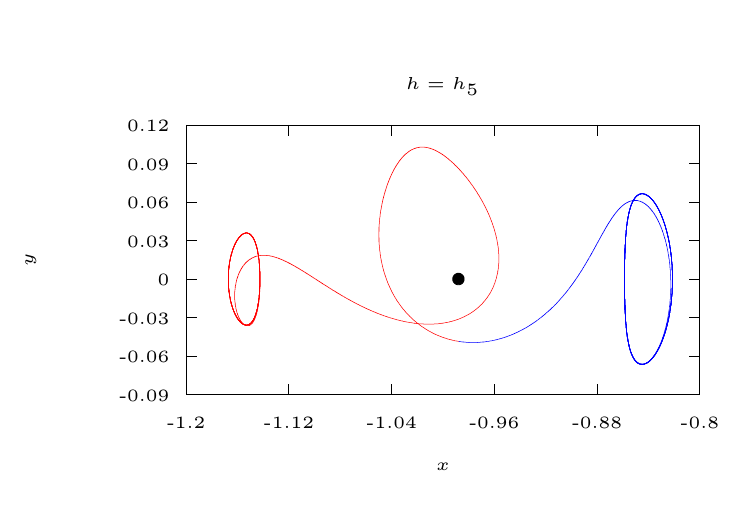}
   \includegraphics{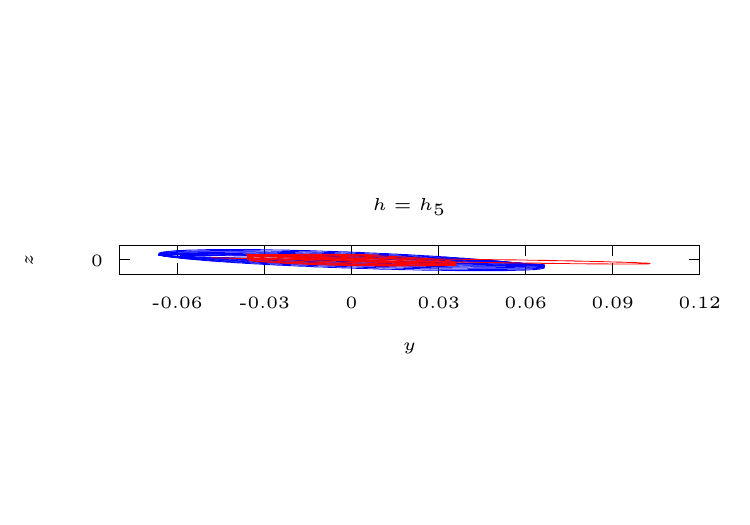}
   \includegraphics{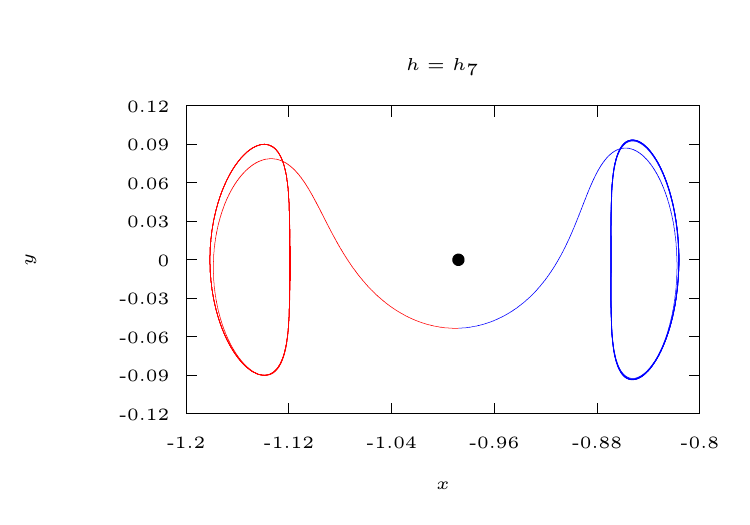}
   \includegraphics{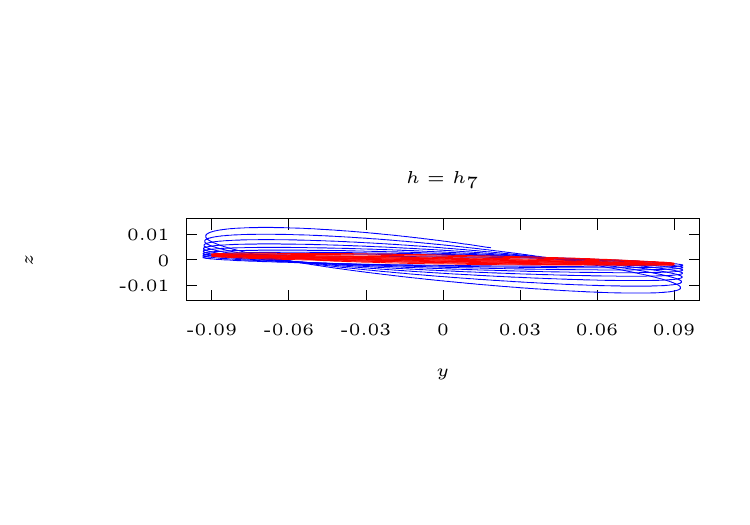}
   \caption{\label{fig:dmin}
   Heteroclinic connection from $\W^{c}(L_2)$ to $\W^{c}(L_1)$ with minimum
   distance to the second primary along the trajectory at different energy levels
   $h_i$, $i=3,5,7$. 
   Red: unstable branch of $L_2$.  Blue: stable branch of $L_1$. 
}
\end{figure}

\begin{figure}[tbp]
   \centering
   \includegraphics{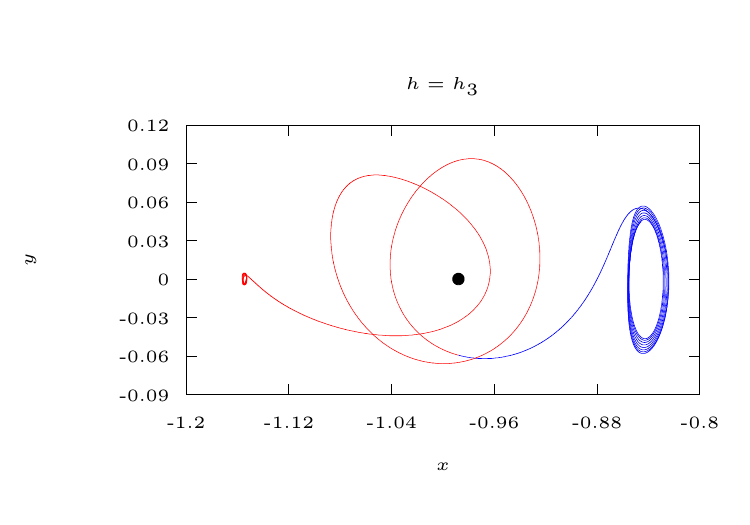}
   \includegraphics{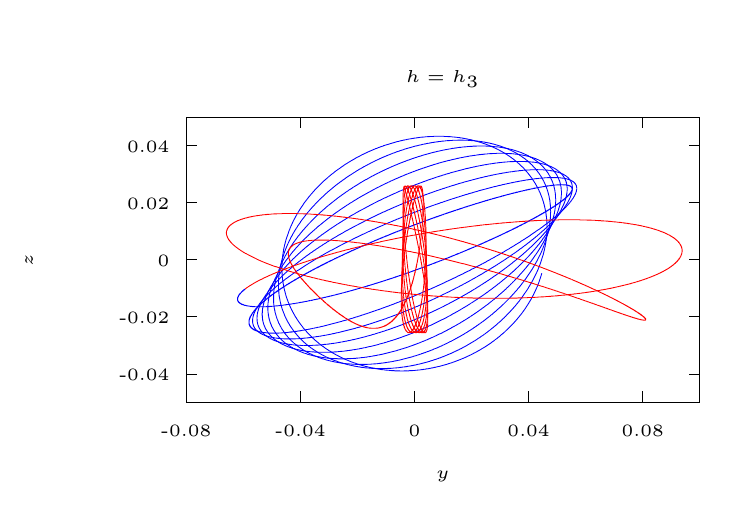}
   \includegraphics{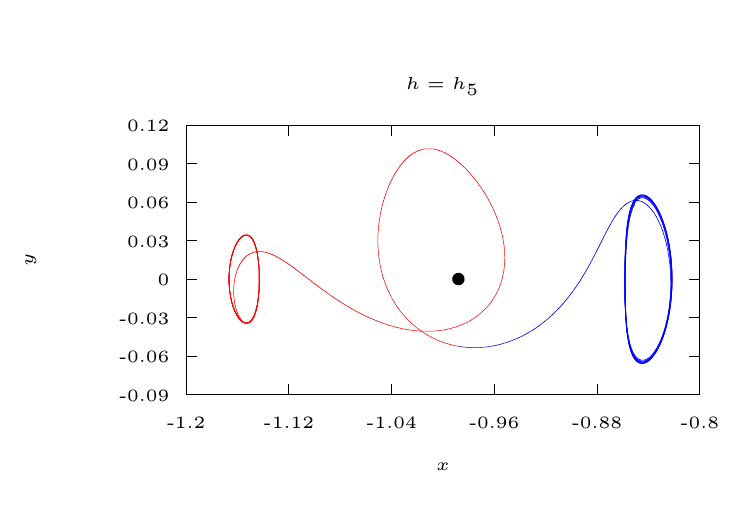}
   \includegraphics{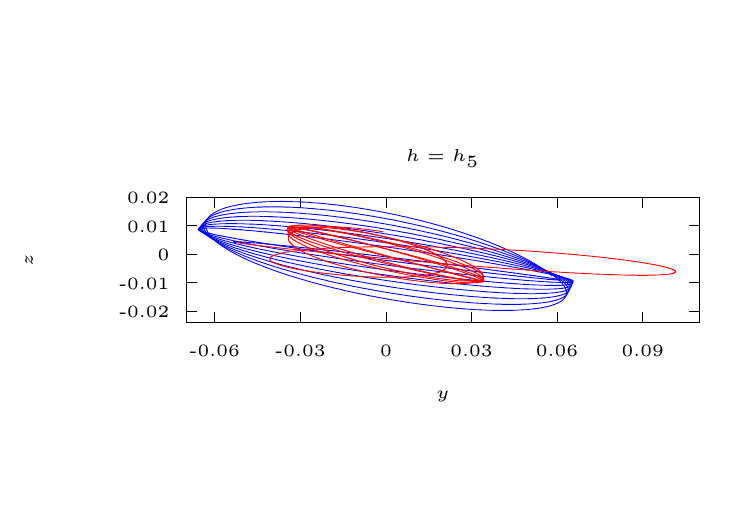}
   \includegraphics{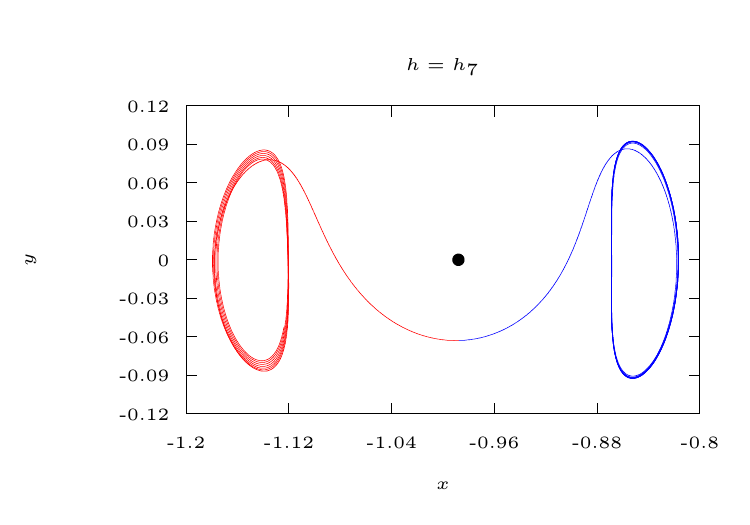}
   \includegraphics{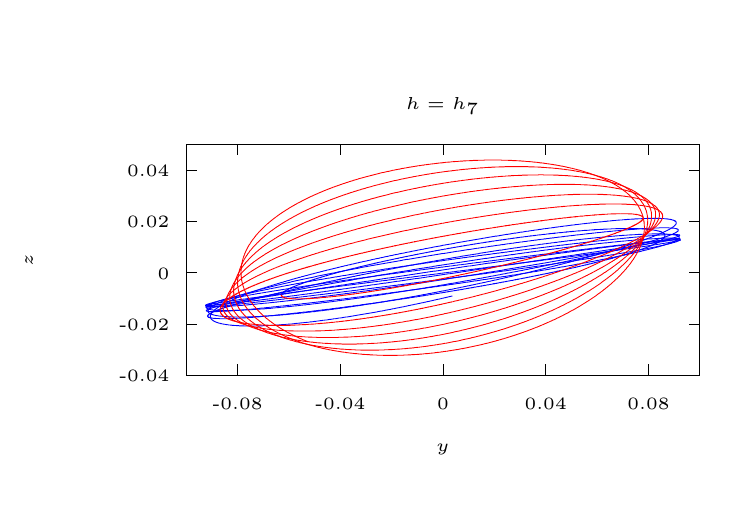}
   \caption{\label{fig:zmax} 
      Heteroclinic connection from $\W^{c}(L_2)$ to $\W^{c}(L_1)$ with maximum
   $z$-amplitude along the trajectory at different energy levels $h_i$,
   $i=3,5,7$.
   Red: unstable branch of $L_2$.
   Blue: stable branch of $L_1$. 
}
\end{figure}

\section{Conclusions}

In this paper we have shown how to compute whole sets of heteroclinic
connections between iso-energetic slices of center manifolds of fixed points of 
center$\times$center$\times$saddle type of autonomous, 3-degrees of-freedom
Hamiltonians. To do so, we have used the parameterization method by
adding a new extra style to uncouple the hyperbolic part from the central and
thus making explicit
the fibered structure of center-stable and center-unstable manifolds. A
meshing strategy for iso-energetic slices of center manifolds that avoids
numerical integration of the reduced equations is crucial for the whole
procedure to be computationally efficient.

This methodology is applied to the Earth-Moon spatial, circular RTBP to obtain
whole sets of heteroclinic connections from the center manifold of $L_2$ to
the center manifold of $L_1$ for different energy levels. As explained, these
sets contain a full description of the connections between different objects
from the departure center manifold to the arrival center manifold and that is
why we expect these sets to be useful in preliminary mission design.
Some comments in this direction have been made.

The applicability of the procedure presented in this paper is limited by the
validity of the expansions of center-stable and center-unstable manifolds. In
this paper, for the Earth-Moon case, they have been checked to be accurate up
to energies in which the $L_1$ Halo family of periodic orbits has already
appeared and has a certain non-small amplitude
(Figures~\ref{fig:errbyorder},\ref{fig:secp}).  A starting point in order to
go to higher energies would be to numerically globalize the center-stable and
center-unstable manifolds of the libration points. This can be done from a
sufficiently dense grid of
numerically computed invariant tori together with their stable and unstable
manifolds \cite{2009MoBaGoOll,2012MoBaGoOll}. There are several methodologies
available for the numerical computation of invariant tori and their
manifolds \cite{2001GoMo,2018BaOlSche,2021aHaMo,2022KuAndLla}, of which
the ones based on the parameterization method and flow maps (or stroboscopic
maps) seem to be the most efficient computationally.

\section*{Acknowledgements}

This work has been supported by the Spanish grants MCINN-AEI
PID2020-118281GB-C31 and PID2021-125535NB-I00, the Catalan grant 2017 SGR
1374, by the Spanish State Research Agency, through the Severo Ochoa and
Mar{\'\i}a de Maeztu Program for Centers and Units of Excellence in R\&D
(CEX2020-001084-M), the European Union's Horizon 2020 research and innovation
program under the Marie Sklodowska-Curie grant agreement No.~734557, the
Secretariat for Universities and Research of the Ministry of Business and
Knowledge of the Government of Catalonia, and by the European Social Fund.

\bibliographystyle{abbrv}
\bibliography{references.bib}

\end{document}